\documentclass[12pt]{amsart} 

\usepackage{amssymb}
\usepackage{amsfonts}
\usepackage{amsmath}

\usepackage{epsfig}
\usepackage{amstext}
\usepackage{amsthm}
\usepackage{hyperref}
\usepackage{bm}
\usepackage{bbm}

\usepackage{float}
\numberwithin{equation}{section}

\usepackage{a4wide}
\usepackage{hyperref}

\title{Parallelizing Spectral Algorithms for Kernel Learning}
\author{Gilles Blanchard
        \and Nicole M\"{u}cke \\
        }

\address{Institute of Mathematics, University of Potsdam, Karl-Liebknecht-Straße 24-25 14476 Potsdam, Germany}
\email{\{blanchard,nmuecke\}@uni-potsdam.de}

\date{\today}

\theoremstyle{plain}

\newtheorem{theo}{Theorem}[section]

\newtheorem{lem}[theo]{Lemma}

\newtheorem{prop}[theo]{Proposition}

\newtheorem{cor}[theo]{Corollary}

\theoremstyle{definition}

\newtheorem{defi}[theo]{Definition}

\theoremstyle{remark}
\newtheorem{example}[theo]{Example}

\newcommand{\cal}{\mathcal}

\newcommand{\E}{{\mathbb{E}}}

\newcommand{\R}{{\mathbb{R}}}
\newcommand{\N}{{\mathbb{N}}}

\newcommand{\lam}{\lambda}

\newcommand{\h}{{\cal H}_K}

\newcommand{\hhh}{{\cal H}}
\newcommand{\X}{{\cal X}}

\newcommand{\prf}{\begin{proof}} 
\newcommand{\prfend}{\end{proof}} 
\newcommand{\Y}{{\cal Y}}
\newcommand{\la}{\langle}
\newcommand{\ra}{\rangle}

\newcommand{\x}{{\bf x}}
\newcommand{\y}{{\bf y}}

\newcommand{\z}{{\bf z}}

\newcommand{\M}{{\cal M}}
\newcommand{\B}{{\cal B}}
\newcommand{\A}{{\cal A}}

\newcommand{\NN}{{\cal N}}
\newcommand{\PPP}{{\cal P}}
\newcommand{\eps}{\varepsilon}

\newcommand{\etainv}{\eta^{-1}}

\newcommand{\sumup}{\frac{1}{m}\sum_{j=1}^m}

\newcommand{\eigv}{\mu}

\newcommand{\Ran}{\mathrm{Im}}

\newcommand{\nux}{\nu}

\newcommand{\fo}{f_{\rho}}

\newcommand{\fest}{\widehat{f}_n}

\newcommand{\priorle}{\PPP^<}

\newcommand{\hs}{{\mathrm{HS}}}
\newcommand{\tr}{\mathrm{tr}}

\newcommand{\paren}[1]{\left(#1\right)}

\newcommand{\inner}[1]{\left\langle#1\right\rangle}
\newcommand{\norm}[1]{\left\|#1\right\|}
\newcommand{\snorm}[1]{\left\| B^s(#1)\right\|}

\newcommand{\abs}[1]{\left\lvert #1 \right\rvert}

\RequirePackage{color}

\newcounter{nbdrafts}
\makeatletter
\newcommand{\checknbdrafts}{
\ifnum \thenbdrafts > 0
\@latex@warning@no@line{**********************************************************************}
\@latex@warning@no@line{* The document contains \thenbdrafts \space draft note(s)}
\@latex@warning@no@line{**********************************************************************}
\fi}

\makeatother

\newcommand{\beq}{\begin{equation}}
\newcommand{\eeq}{\end{equation}}

\newcommand{\Err}{\mbox{Err}}

\begin{document}


\begin{abstract}
We consider a 
distributed learning approach in supervised learning for a large class of spectral regularization methods in an RKHS framework. 
The data set of size n is partitioned into 
$m=O(n^\alpha)$ disjoint subsets. On each subset, some spectral regularization method (belonging to a large class, including in particular Kernel Ridge Regression, $L^2$-boosting and spectral cut-off) is applied. The regression function $f$ is then estimated via simple averaging, leading to a substantial reduction in computation time. We show that minimax optimal rates of convergence are preserved if m grows sufficiently slowly (corresponding to an upper bound for $\alpha$) as $n \to \infty$, depending 
on the smoothness assumptions on $f$ and the intrinsic dimensionality. In spirit, our approach is classical.   
\end{abstract}

\maketitle

\section{Introduction}

Distributed learning (DL) algorithms are a standard tool for saving computation time in machine learning problems where massive datasets are involved: 
Dividing randomly  data of cardinality $n$ into $m$ equally-sized , easy manageable partitions and evaluating them in parallel roughly gains a factor  $m^{-2}$ (for time and memory) compared to the single machine approach. The final output is obtained from averaging the individual outputs\footnote{For the sake of simplicity, throughout this paper we assume that $n$ is divisible by $m$. This could always be achieved be disregarding some data; alternatively, it is straightforward to show that admitting one smaller block in the partition does not affect the asymptotic results of this paper. We shall not try to discuss this point in greater detail. In particular, we shall not analyze in which general framework our simple averages could be replaced by weighted averages.}. 
\\
Recently, DL was studied in several machine learning contexts: in point estimation \cite{Li13}, matrix factorization \cite{Jordan11},  smoothing spline models and testing \cite{ChengShang16}, local average regression \cite{Chang16}, in classification (kernel SVMs \cite{HsiehSiD13} and feature space decomposition \cite{Guo15}) and also in kernel (ridge) regression (KRR) \cite{ZhangDuchWai15}, \cite{Zhou15}, \cite{XuZhangLi15}. 
\\
In this paper, we study the DL approach for the statistical learning problem 
\begin{equation}
  \label{eq:learnmodel}
  Y_i:=f(X_j) + \eps_i \; , j=1\,,\ldots,n\,,
  \end{equation}
at random i.i.d. data points $X_1,\ldots,X_n$\, drawn according to a probability distribution $\nux$ on $\X$, where $\eps_j$ are independent centered noise variables. The unknown regression function $f$ is real-valued and belongs to some reproducing kernel Hilbert space with bounded kernel $K$. 
We partition the given data set $D=\{(X_1, Y_1), ..., (X_n, Y_n)\} \subset \X \times \R$ into $m$ disjoint equal-size subsets $D_1, ..., D_m$. On 
each subset $D_j$, we compute a local estimator $\hat f^{\lam}_{D_j}$, using a spectral regularization method. The final estimator for the target function $f$ is obtained by simple averaging: $\bar f^{\lam}_{D}: = \sumup \hat f^{\lam}_{D_j}$. 
\\
The non-distributed setting (m=1) has been studied in the recent paper \cite{BlaMuc16}\;, building the root position of our results in the distributed setting, where (weak and strong) minimax optimal rates of convergence are established.  Our aim is to extend these results to distributed learning and to derive minimax optimal rates. We again apply a fairly large class of spectral regularization methods, including the popular KRR, $L^2$-boosting and spectral cut-off. As in \cite{BlaMuc16}\;, we let 
$T: f \in \hhh_K \mapsto \int f(x) K(x,\cdot) d\nux(x) \in \hhh_K$
denote the kernel integral operator associated to $K$ and the sampling measure $\nux$. Our rates of convergence
are governed by a {\em source condition} assumption on $f$ of the form
$\norm{T^{-r} f} \leq R$ for some constants $r,R>0$ as well as by the {\em ill-posedness} of the problem,
as measured by an assumed power decay of the eigenvalues of $T$ with exponent $b>1$\,. 
We show, that for 
$s \in [0,\frac{1}{2}]$ in the sense of p-th moment expectation
\begin{equation}
\label{eq:result}
\norm{T^s(f -\bar f^{\lam_n}_{D})}_{\h}  \lesssim R \paren{\frac{\sigma^2}{R^2n}}^{\frac{(r+s)}{2r+1+1/b}}\,,
\end{equation}
for an appropriate choice of the regularization parameter $\lam_n$\,, depending on the global sample size $n$ as well as on $R$ and 
the noise variance $\sigma^2$ (but not on the number $m$ of subsample sets).   
Note that $s=0$ corresponds to the reconstruction error (i.e. $\h$- norm), and
$s=\frac{1}{2}$ to the prediction error (i.e., $L^2(\nux)$ norm).  
The symbol $\lesssim$ means that the inequality holds up to a multiplicative constant that can depend on various
parameters entering in the assumptions of the result, but not on $n$, $m$, $\sigma$, nor $R$\,. 
An important assumption is that the inequality $q \geq r+s$ should hold, where $q$ is the
{\em qualification} of the regularization method, a quantity defined in the classical theory
of inverse problems (see Section~\ref{se:regulariz} for a precise definition)\,. 
Basic problems are the choice of the regularization parameter on the subsamples and, most importantly, the proper choice of $m$,
since it is well known that choosing $m$ too large gives a suboptimal convergence rate in the limit $n \to \infty$, see e.g. \cite{XuZhangLi15}.

Our approach to this problem is classical. Using a bias-variance decomposition and choosing the regularization parameter according to the total 
sample size $n$  yields undersmoothing on each of the $m$ individual samples. The bias estimate is then straightforward. For the hard part we write the variance as a sum of independent random variables, leading to a substantial reduction of variance by averaging.  

To the best of our knowledge, comparable results up to completion of this article had been  restricted to KRR, corresponding to Tikhonov regularization.  
In \cite{ZhangDuchWai15} the authors derive Minimax-optimal rates in 3 cases
(finite rank kernels, sub- Gaussian decay of eigenvalues of the kernel and polynomial decay), provided m satisfies a certain upper bound, depending on the rate of decay of the eigenvalues and an additional crucial upper bound on the eigenfunctions $\phi_j$ of the Mercer kernel (see Section \ref{discussion}).  
It is therefore of great interest to investigate if and how $m$ can be allowed to go to infinity as a function of $n$ without imposing any conditions on the eigenfunctions of the kernel. 
Results in this direction have been obtained in the recent paper \cite{Zhou15}, for KRR, which is a great improvement on the worst rate of \cite{ZhangDuchWai15}. The authors dub their approach {\em a second order decomposition}, which uses concentration inequalities and certain resolvent identities adapted to KRR. 
After this paper had been completed, however, we learned of the Oberwolfach report \cite{Zhou_oberwolfach}, where the authors have reported results for general spectral regularization methods, which are similar to the results in this paper. At the time of writing, we are not aware of any published proof. 
It is unclear to us how the authors of \cite{Zhou_oberwolfach} prove their results. 
They require bounded output space, a continuous kernel (ours need only be bounded)and their estimates are only in $L^2-$ sense, not in RKHS-norm. 
Furthermore, they do not seem to track the dependence on the noise variance and the source condition as precisely as we do. For more detail, we refer to our Discussion in Section 4\;.

The outline of the paper is as follows.  Section 2 contains notation and the setting.  Section 3 states our main result on distributed learning. Section 4 presents 
numerical studies, followed by a concluding discussion and a more detailed comparison of our results in Section 5. 
In Section 6 we prove our theorems.
  

\section{Notation, Statistical model and distributed learning Algorithm}
\label{se:notation}

In this section, we specify the mathematical background and the statistical model for (distributed) regularized learning. 
We have included this section for self sufficiency and reader convenience. It essentially repeats the setting in \cite{BlaMuc16} in summarized form.  

\subsection{Kernel-induced operators}

We assume that the input space $\X$ is a standard Borel space 
endowed with a probability measure $\nux$\;, the output space is equal to $\R$.  
We let $K$ be a positive semidefinite kernel on $\X \times \X$ which is bounded by $\kappa$. The associated reproducing kernel Hilbert space will be denoted by $\h$.
It is assumed that all functions $f \in \h$ are measurable and bounded in
supremum norm, i.e. $\norm{f}_\infty \leq \kappa \norm{f}_{\h}$ for all
$f \in \h$. Therefore, $\h$ is a subset of $L^2(\X,\nux)$\,, with $S: \h  \longrightarrow L^2(\X, \nux)$ being the inclusion operator, satisfying  $\norm{S}\leq \kappa$\,.
The adjoint operator $S^{*}: L^2(\X, \nux)\longrightarrow \h $ is identified as  
\[ S^{*}g=\E_{\sim \nu}[g(X)K_{X}] = \int_{\X} g(x)K_x\;\nu(dx)  \;. \] 
Setting $T_x=K_{x}\otimes K_x^{*} :\h\longrightarrow \h $,  the covariance operator 
is given by 
\[T = \E_{\sim \nu} [K_{X}\otimes K_X^{*}] =  \int_{\X} \la \cdot , K_x \ra_{\h}   K_x\;\nu(dx)\; , \] 
which can be shown to be positive self-adjoint trace class (and hence is compact). 
The corresponding empirical versions of these operators are given by 
\[S_{\x}: \h \longrightarrow \R^n \;, \qquad (S_{\x}f)_j= \la f, K_{x_j} \ra_{\h} \;, \]
\[S^*_{\x}:\R^n   \longrightarrow \h  \;, \qquad S_{\x}^{*}\y = \frac{1}{n}\sum_{j=1}^n y_j K_{x_j}\,, \]
\[T_{\x}:= S^*_{\x}S_{\x}: \h \longrightarrow \h \; , \qquad T_{\x} =\frac{1}{n}\sum_{j=1}^n  K_{x_j} \otimes  K_{x_j}^{*}  \; . \] 
We introduce the shortcut notation $\bar T = \kappa^{-2} T$ and $\bar T_{\x} := \kappa^{-2} T_{\x}$\;, ensuring $||\bar T|| \leq 1$ and  $||\bar T_x|| \leq 1$.   Similarly, 
$\bar S = \kappa^{-1} S$ and $ \bar S_{\x_j} := \kappa^{-2} S_{\x_j}$\,, ensuring $||\bar S|| \leq 1$ and $||\bar S_x|| \leq 1$.  
The numbers $\eigv_j$ are the positive eigenvalues of $\bar T$ satisfying 
$0< \eigv_{j+1} \leq \eigv_j$ for all $j>0$\, and $\eigv_j \searrow 0$.

\subsection{Noise assumption and prior classes}
In our setting of kernel learning, the sampling is assumed to be
random i.i.d., where each observation point $(X_i,Y_i)$ follows the model
$ Y = f(X) + \eps \,.$
For $(X,Y)$ having distribution $\rho$,
we assume: The conditional expectation wrt. $\rho$ of $Y$ given $X$ exists and it holds
for $\nux$-almost all $x \in X$\,:
\begin{equation}
  \label{basicmodeleq}
  \E_\rho[Y | X=x] = \fo(x) \,, \text{ for some } \fo \in \h\,.
\end{equation}
Furthermore, we will make the following assumption on the 
observation noise distribution: There exists $\sigma > 0$ such that 
\begin{equation}
\label{bernstein}
\E[\; \abs{Y - \fo(X)}^{2} \; | \; X \;] \leq  \sigma^2  \quad \nux - {\rm a.s.} \;.
\end{equation}
To derive nontrivial rates of convergence, we concentrate
our attention on specific subsets (also called {\em models})  of the class of probability measures.  
If $\PPP$ denotes the set of all probability distributions on $\X$, 
we define classes of sampling distributions by introducing decay conditions on
the eigenvalues $\eigv_i$ of the operator $T_\nux$. 
For $b>1$ and $\beta >0$\,, we set
\[ \priorle(b, \beta):=\{ \nu \in  {\PPP}:  \; \eigv_j \leq \beta/j^b \; \; \forall j \geq 1 \}  \; , \]
For a subset $\Omega \subseteq \h$, we let ${\cal K}(\Omega)$ be the set of regular conditional probability distributions
$\rho(\cdot|\cdot)$ on ${\cal B}(\R)\times \X$ such that
$(\ref{basicmodeleq})$ and $(\ref{bernstein})$ hold for some $\fo \in \Omega$.
We will focus on a {\it H\"older-type source condition}, i.e. 
given $r>0, R>0$ and $\nu \in {\cal P}$, we define
\begin{equation}
\label{sourceset}
\Omega_{\nux}(r,R) := \{f \in \h : f = \bar T_{\nux}^rh,\; \norm{h}_{\h}\leq 
R   \}.
\end{equation}
Then the class of models which we will consider will be defined as
\begin{equation}
\label{measureclass}
 \M(r,R,\PPP') \; := \; \{ \; \rho(dx,dy)=\rho(dy|x)\nux(dx)\; : \; 
\rho(\cdot|\cdot)\in {\cal K}(\Omega_{\nux}(r,R)), \; \nux \in  \PPP' \;\} \;,
\end{equation}
with $\PPP'=\priorle(b,\beta)$. 
As a consequence, the class of models depends not only on the smoothness properties of the solution (reflected in the parameters $R>0, \; r>0$), 
but also essentially on the decay of the eigenvalues of $\bar T_\nux$.



\subsection{Regularization}
\label{se:regulariz}

In this subsection, we introduce the class of linear regularization methods based on spectral theory for 
self-adjoint linear operators. These are standard methods for finding stable solutions for ill-posed inverse 
problems.  Originally, these methods were developed in the deterministic context, see  \cite{engl}. Later on, they have been applied
to probabilistic problems in machine learning, see  \cite{rosasco} or \cite{BlaMuc16}.

\begin{defi}[\protect{\bf Regularization function}]
\label{regudef}
Let $g: (0,1]\times [0, 1] \longrightarrow \R$ be a function and write 
$g_{\lam}=g(\lam, \cdot)$. The family $\{g_{\lam}\}_{\lam}$ is called 
{\it regularization function}, if the following conditions hold: 
\begin{enumerate}
\item[(i)]
There exists a constant $D'<\infty$ such that for any $0 < \lam \leq 1$
\begin{equation*}
\sup_{0<t\leq 1}|tg_{\lam}(t)| \leq D'\;.
\end{equation*}

\item[(ii)]
There exists a constant $E<\infty$ such that for any $0 < \lam \leq 1$
\begin{equation}
  \label{eq:supg}
\sup_{0<t\leq 1}|g_{\lam}(t)| \leq \frac{E}{\lam}\;.
\end{equation}

\item[(iii)]
Defining the {\em residual} $r_{\lam}(t):= 1-g_{\lam}(t)t$\,, there exists a constant $\gamma_0 <\infty$ such that for any $0 < \lam \leq 1$
\begin{equation*}
\sup_{0<t\leq 1}|r_{\lam}(t)| \leq \gamma_0 \;.
\end{equation*}
\end{enumerate}
\end{defi}

It has been shown in e.g. \cite{DicFosHsu15}, \cite{BlaMuc16} that attainable learning rates are essentially linked with the qualification of the regularization $\{g_{\lam}\}_{\lam}$, being the maximal $q$ such that for any $0<\lam\leq 1$
\begin{equation}
\label{eq:quali}
\sup_{0<t\leq 1}|r_{\lam}(t)|t^{q} \leq \gamma_{q}\lam^{q}.
\end{equation}
for some constant $\gamma_q>0$\,. The most popular examples include: 


\begin{example} (Tikhonov Regularization, Kernel Ridge Regression)  
\label{tikhexample}
The choice $g_{\lam}(t) = \frac{1}{\lam + t}$ corresponds to {\it Tikhonov regularization}. In this case we have $D'=E=\gamma_0 =1$. 
The qualification of this method is $q =1$ with $\gamma_{q} = 1$.
\end{example}

\begin{example}(Landweber Iteration, gradient descent )  
\label{landweber}
The {\it Landweber Iteration} (gradient descent algorithm with constant stepsize) is defined by
\[ g_{k}(t) = \sum_{j=0}^{k - 1}(1-t)^j \, \mbox{ with $k=1/\lam$ $\in \N$} \; . \]
We have $D'=E=\gamma_0 = 1$. The qualification $q$ of this algorithm can be arbitrary with $\gamma_q=1$ if $0<q\leq 1$ and $\gamma_q=q^q$ if $q>1$. 
\end{example}

\begin{example}{\bf ($\nu$- method) } 
\label{ex:numethod}
The {\it $\nu-$ method} belongs to the class of so called {\it semi-iterative regularization methods}. This method has finite qualification $q=\nu$ 
with $\gamma_q$ a positive constant. Moreover, $D=1$ and $E=2$. The filter is given by $g_k(t) = p_k(t)$, a polynomial of degree $k-1$, with regularization parameter $\lam \sim k^{-2}$, which makes this method much faster as e.g. gradient descent. 
\end{example}



\subsection{Distributed Learning Algorithm}
\label{sec:dist_learn_algo}

We let  $D=\{(x_j, y_j)\}_{j=1}^n \subset \X \times \Y$ be the dataset, which we partition into $m$ disjoint subsets $D_1, ..., D_m$, each having 
size $\frac{n}{m}$.  
Denote the  $jth$  data vector by $(\x_j, \y_j) \in (\X \times \R)^{\frac{n}{m}}$.   
On each subset we compute a local estimator for a suitable a-priori parameter choice $\lam = \lam_n$ according to
\begin{equation}
\label{def:subset_estimator}
f_{D_j}^{\lam_n} :=g_{\lam_n}(\kappa^{-2}T_{\x_j})\kappa^{-2}S^{\star}_{\x_j}\y_j  = 
g_{\lam_n}(\bar T_{\x_j})\bar S^{\star}_{\x_j}\y_j\; .
\end{equation} 
By $f_D^{\lam}$ we will denote the estimator using the whole sample $m=1$\;. 
The final estimator is given by simple averaging the local ones:
\begin{equation}
\label{def:estimator}
\bar f^{\lam}_{D}:= \frac{1}{m}\sum_{j=1}^m f_{D_j}^\lam \; .
\end{equation}



\section{Main Results}


This section presents our main results. Theorem \ref{prop:approx_error} and Theorem \ref{prop:sample_error_exp} contain separate estimates on
the approximation error and the sample error and lead to
Corollary \ref{maintheo:total_averaged} which gives an upper bound for the error $\norm{\bar T^s (\fo - \bar f_{D}^{\lam})}_{\h}$ 
and presents an upper rate of convergence for the sequence of distributed learning algorithms. 

For the sake of the reader we recall Theorem \ref{theo:BlaMuc16}, which was already shown in \cite{BlaMuc16}, presenting the minimax optimal rate for the single machine problem. This yields an estimate on the difference between the single machine and the distributed learning algorithm in Corollary  \ref{cor:rest}.
  
We want to track the precise behavior of these rates not only for what concerns the
exponent in the number of examples $n$, but also in terms of their scaling
(multiplicative constant)
as a function of some important parameters (namely the noise variance $\sigma^2$ and the complexity radius $R$ in the source condition).
For this reason, we introduce a notion of a family of rates over a family of models.
More precisely, 
we consider an indexed family
$(\M_\theta)_{\theta \in \Theta}$\,, where for all $\theta \in \Theta$\,, $\M_\theta$ is a class of Borel probability distributions
on $\X \times \R$ satisfying the basic general assumptions \ref{basicmodeleq} and \eqref{bernstein}.
We consider rates of convergence in the sense of the $p$-th moments of the estimation error, where $1 \leq p  < \infty$ is a fixed real number.


As already mentioned in the Introduction, our proofs are based on a classical bias-variance decomposition as follows: Introducing
\begin{equation}
\label{def:tildef}
\tilde f^{\lam}_D =\sumup g_{\lam}(\bar T_{\x_j})\bar T_{\x_j}\fo \;, 
\end{equation}
we write
\begin{align}
\label{eq:decompo}
\bar T^s( \fo - \bar f_D^{\lam} )&= \bar T^s(\;\fo - \tilde f^{\lam}_D \;) +  \bar T^s( \;\tilde f^{\lam}_D - \bar f_D^{\lam} \; )  \nonumber \\
                              &=   \underbrace{\sumup \bar T^sr_{\lam}( \bar T_{\x_j})\fo}_{Approximation \; Error} + 
   \underbrace{\sumup \bar T^sg_{\lam}(\bar T_{\x_j})(\bar T_{\x_j}\fo- S^{*}_{ \x_j }\y_j )}_{ Sample \; Error} \;. \\
   & \nonumber
\end{align}




In all the forthcoming results in this section, we let $s \in [0,\frac{1}{2}]$, $p\geq 1$ and consider the model $\M_{\sigma,M,R}:={\M}(r,R, {\priorle}(b, \beta))$\,
where $r>0$, $b>1$ and $\beta>0$ are fixed, and $\theta=(R,M,\sigma)$ varies in $\Theta=\R^3_+$.
Given a sample $D \subset (\X \times \R)$ of size $n$, define $\bar f_{D}^{\lam_n}$, $f_{D}^{\lam_n}$ as in Section \ref{sec:dist_learn_algo} and $ \tilde f^{\lam_n}_D$ as in \eqref{def:tildef},  
using a regularization function of qualification $q\geq r+s$, with  parameter sequence 
\begin{equation}
  \label{eq:choicelam:averaged}
   \lam_n:=\lam_{n,(\sigma,R)} := \min\paren{ \paren{ \frac{\sigma^2}{R^2 n}}^{\frac{b}{2br+b + 1}},1} \; ,
\end{equation}
independent on $M$. Define the sequence 
\begin{equation}
\label{def:rate}
 a_n:= a_{n,(\sigma,R)}:=  R \paren{\frac{\sigma^2}{R^2n}}^{\frac{b(r+s)}{2br+b+1}} \;. 
\end{equation} 

We recall from the introduction that we shall always assume that $n$ is a multiple of  $m$.
With these preparations, our main results are:



\begin{theo}[Approximation Error]
\label{prop:approx_error}
If the number $m$ of subsample sets satisfies 
\begin{equation}
\label{cond:m_alpha2}
m\leq n^{\alpha} \;, \quad \alpha < \frac{2b\min\{ r,  1 \}}{2br+b+1}   \;,
\end{equation}
Then
\[
   \sup_{(\sigma,M,R) \in \R_+^3}  \limsup_{n \rightarrow \infty} \sup_{\rho \in \mathcal{M}_{\sigma,M,R}} \frac{   \E_{\rho^{\otimes n}}
     \Big[\big\|\bar T^s( \fo - \tilde f_{D}^{\lam_{n}})\big\|^p_{\h}\Big]^{\frac{1}{p}}}{a_{n}} < \infty   \,.
\]
\end{theo}


\begin{theo}[Sample Error]
\label{prop:sample_error_exp}
If the number $m$ of subsample sets satisfies 
\begin{equation}
m\leq n^{\alpha} \;, \quad \alpha < \frac{2br}{2br+b+1}   \;,
\end{equation}
Then
\[
    \sup_{(\sigma,M,R) \in \R_+^3} \limsup_{n \rightarrow \infty} \sup_{\rho \in \mathcal{M}_{\sigma,M,R}} \frac{   \E_{\rho^{\otimes n}}
     \Big[\big\| \bar T^s( \tilde f^{\lam_n}_D - \bar f_D^{\lam_n} )   \big\|^p_{\h}\Big]^{\frac{1}{p}}}{a_{n}}  < \infty  \,.
\]
\end{theo}

And, as consequence (by \eqref{eq:decompo} and applying the triangle inequality):

\begin{cor}
\label{maintheo:total_averaged}
If the number $m$ of subsample sets satisfies 
\begin{equation}
\label{cond:m_alpha_final}
m\leq n^{\alpha} \;, \quad \alpha < \frac{\min\{ 2br,  b+1 \}}{2br+b+1}   \;,
\end{equation}
then 
the sequence \eqref{def:rate} is an upper rate of convergence in $L^p$, for the interpolation
norm of parameter $s$, for the sequence of estimated solutions $(\bar f_{D}^{\lam_{n,(\sigma,R)}})$ over 
the family of models $(\M_{\sigma,M,R})_{(\sigma,M,R) \in \R_+^3}$\,, i.e.
\[
 \sup_{(\sigma,M,R) \in \R_+^3}  \limsup_{n \rightarrow \infty} \sup_{\rho \in \mathcal{M}_{\sigma,M,R}} \frac{   \E_{\rho^{\otimes n}}
     \Big[\big\|\bar T^s( \fo -  \bar f_{D}^{\lam_{n}})\big\|^p_{\h}\Big]^{\frac{1}{p}}}{a_{n}}  <  \infty   \,.
\]
\end{cor}

\begin{theo}[Blanchard, M\"ucke (2017) \cite{BlaMuc16}]
\label{theo:BlaMuc16}
The sequence \eqref{def:rate} 
is an upper rate of convergence in $L^p$ for all $p\geq1$, for the interpolation
norm of parameter $s$, for the sequence of estimated solutions $(f_{D}^{\lam_{n,(\sigma,R)}})$ - independent on $M$ - over 
the family of models $(\M_{\sigma,M,R})_{(\sigma,M,R) \in \R_+^3}$\,, i.e.
\[
   \sup_{(\sigma,M,R) \in \R_+^3} \limsup_{n \rightarrow \infty} \sup_{\rho \in \mathcal{M}_{\sigma,M,R}} \frac{   \E_{\rho^{\otimes n}}
     \Big[\big\| \bar T^s( \fo - f_{D}^{\lam_{n}})\big\|^p_{\h}\Big]^{\frac{1}{p}}}{a_{n}}  <  \infty   \,.
\]
\end{theo}

Combining Corollary \ref{maintheo:total_averaged}\; with Theorem \ref{theo:BlaMuc16}\; by applying the triangle inequality immediately yields:

\begin{cor}
\label{cor:rest}
If the number $m$ of subsample sets satisfies 
\begin{equation}
\label{cond:m_alpha}
m\leq n^{\alpha} \;, \quad \alpha < \frac{2b\min\{ r,  1 \}}{2br+b+1}   \;,
\end{equation}
then 
\[
    \sup_{(\sigma,M,R) \in \R_+^3} \limsup_{n \rightarrow \infty} \sup_{\rho \in \mathcal{M}_{\sigma,M,R}} \frac{   \E_{\rho^{\otimes n}}
     \Big[\big\| \bar T^s( f^{\lam_n}_D - \bar f_D^{\lam_n} )   \big\|^p_{\h}\Big]^{\frac{1}{p}}}{a_{n}} < \infty  \,.
\]
\end{cor}


\section{Numerical Studies}
\label{sec:numerical_studies}

In this section we numerically study  the error in $\h$- norm, corresponding to $s=0$ in Corollary \ref{maintheo:total_averaged} (in expectation with 
$p=2$) both in the single machine and distributed learning setting. 
Our main interest is to study the upper bound for our theoretical exponent $\alpha$, parametrizing the size of subsamples in terms of the total sample size, $m =n^\alpha$, in different smoothness regimes. In addition we shall demonstrate in which way parallelization serves as a form of regularization.
 
More specifically, we let $\h=H_0^1[0,1]$ with kernel $K(x,t)= x\wedge t-xt$. 
For all experiments in this section, we simulate data from the regression model
\[  Y_i = \fo (X_i) + \epsilon_i \;, \qquad \; i=1,...,n\;,  \] 
where the input variables $X_i \sim Unif[0,1]$ are uniformly distributed and the noise variables $\eps_i \sim N(0,\sigma^2) $ are normally distributed with standard deviation  $\sigma=0.005$. 
We choose the target function $\fo$ according to two different cases, namely $r<1$ ({\it low smoothness}) and $r=\infty$ ({\it high smoothness}). 
To accurately determine the degree of smoothness $r>0$, we apply Proposition \ref{prop:fourier}  below by explicitly calculating the Fourier coefficients $(\inner{ \fo , e_j}_{\h})_{j \in \N}$, where $e_j(x) = \frac{\sqrt{2}}{\pi j}\cos(\pi j x)$, for $j \in \N^*$,  forms an ONB of $\h$. Recall that the rate of eigenvalue decay is explicitly given by  $b=2$, meaning that we have full control over all parameters in \eqref{cond:m_alpha}.
From \cite{engl} we need

\begin{prop}
\label{prop:fourier}
Let $\h, {\cal H}_2$ be separable Hilbert spaces and $S: \h \longrightarrow {\cal H}_2$ be a  compact linear operator with singular system $\{\sigma_j, \varphi_j, \psi_j \}$\footnote{i.e., the $\varphi_j$ are the normalized eigenfunctions of $S^* S$ with eigenvalues $\sigma_j^2$ and $\psi_j= S \varphi_j / ||S \varphi_j||$;
thus $S= \sum \sigma_j \la \varphi_j, \cdot \ra \psi_j$}. Denoting by $S^{\dagger}$ the generalized inverse \footnote{the unique unbounded linear operator
with domain $\Ran (S) \oplus (\Ran (S))^\perp$ in $ {\cal H}_2$ vanishing on $(\Ran (S))^\perp$ and satisfying $S S^{\dagger}= 1$ on $ \Ran(S) $, with range orthogonal to the null space $N(S)$} of $S$, one has for any $r>0$ and $g \in {\cal H}_2$:

$g$ is in the domain of $S^{\dagger}$ and  $S^{\dagger}g \in \Ran((S^{*}S)^r)$ if and only if 
\[  \sum_{j=0}^{\infty}  \frac{|\inner{g, \psi_j}_{{\cal H}_2}|^2}{\sigma_j^{2+4r}} \; < \; \infty \;. \]
\end{prop}

In our case, $\h$ is as above, ${\cal H}_2$ is $L^2([0,1])$ with Lebesgue measure and $S: H_0^1[0,1] \to L^2([0,1])$ is the inclusion. 
Since $H_0^1[0,1]$ is dense in $L^2([0,1])$,  we know that $(\Ran (S))^\perp$ is trivial, giving $S S^{\dagger} =1$ on $\Ran (S)$. Furthermore, $\varphi_j=e_j$ is a normalized eigenbasis of 
$T=S^*S$ with eigenvalues $\sigma_j^2  = (\pi j)^{-2}$. With $\psi_j=\frac{S\varphi_j}{||S\varphi_j||_{L^2}}$
we obtain for $f \in H_0^1[0,1]$
\[
\la Sf, \psi_j \ra_{L^2} = \la Sf, \frac{S e_j}{||S e_j||} \rangle_{L^2} = \la f, \frac{S^*S e_j}{||S e_j||} \ra_{H_0^1} =\sigma_j \la f, e_j \ra_{H_0^1}.
\]
Thus, applying Proposition \ref{prop:fourier} gives 

\begin{cor}
\label{finalfourier}
For $S$ and $T=S^* S$ as above we have for any $r>0$: $f \in \Ran (T^r)$ if and only if
\[
\sum_{j=1}^\infty  j^{4r} |\la f, e_j \ra_{L^2}|^2 < \infty \;.
\]
\end{cor}

Thus, as expected, abstract smoothness measured by the parameter $r$ in the source condition corresponds in this special case to decay of the classical Fourier coefficients which - by the classical theory of Fourier series - measures  smoothness of the periodic continuation  of $f \in L^2([0,1])$ to the real line.

\subsubsection{Low smoothness}

We choose $\fo(x) = \frac{1}{2}x(1-x)$ which clearly belongs to $\h$. A straightforward calculation gives the Fourier coefficient
$\la \fo, e_j \ra = - 2 (\pi j)^{-2}$ for $j$ odd (vanishing for $j$ even). Thus, by the above criterion,
$\fo$ satisfies the source condition $\fo \in Ran(T^r)$ precisely for  $0<r<0.75$\;. 
According to Theorem \ref{theo:BlaMuc16}, the worst case rate in the single machine problem is given by  $n^{-\gamma}$, with $\gamma = 0.25$\;. Regularization 
is done using the $\nu-$ method (see Example \ref{ex:numethod}), with qualification $q= \nu =1$. Recall that the stopping index $k_{stop}$ serves as the regularization parameter $\lambda$, where $ k_{stop} \sim  \lam^{-2}$.  We consider sample sizes from $500,...9000$.  
In the model selection step, we estimate the performance of different models and choose the {\it oracle stopping time} $\hat k_{oracle}$ by minimizing the reconstruction error: 
\[ \hat k_{oracle} = \mbox{arg}\min_{k}  \paren{  \frac{1}{M} \sum_{j=1}^M \norm{ \fo- \hat f_j^k  }^2_{\h} }^{\frac{1}{2}}   \]
over $M=30$ runs. 

In the model assessment step,  
we partition the dataset into $m\sim n^{\alpha}$ subsamples, for any  
$\alpha \in \{ 0, 0.05, 0.1, ..., 0.85 \}$. On each subsample we regularize using the oracle stopping time 
$\hat k_{oracle}$ (determined by using the whole sample). Corresponding to Corollary \ref{maintheo:total_averaged}, the accuracy should be comparable to the one using the whole sample as long as $\alpha < 0.5$\;. In Figure \ref{fig:one} (left panel) we plot the reconstruction error 
$|| \bar f^{\hat k}-\fo ||_{\h}$ versus the ratio $\alpha = \log(m)/\log(n)$ for different sample sizes. We execute each simulation $M=30$ times. The plot supports our theoretical finding.  The right panel shows the reconstruction error versus the total number of samples using 
different partitions of the data. The black curve ($\alpha = 0$) corresponds to the baseline error ($m=0$, no  partition of data).  Error curves below a threshold 
$\alpha < 0.6$ are roughly comparable, whereas curves above this threshold show a gap in performances.  

In another experiment we study the performances in case of (very) different regularization:  Only partitioning the data (no regularization), underregularization (higher stopping index) and overregularization (lower stopping index). The outcome of this experiment amplifies the regularization effect of parallelizing.  Figure \ref{fig:two} shows  the main point:
Overregularization is always hopeless, underregularization is  better. In the extreme case of none or almost none regularization there is a sharp minimum in the reconstruction error which is only slightly larger than the minimax optimal value for the oracle regularization parameter and which is achieved at an attractively large degree of parallelization. Qualitatively, this agrees very well with the intuitive notion that parallelizing serves as  regularization. 

We emphasize that numerical results seem to indicate that parallelization is possible to a slightly larger degree than indicated by our theoretical estimate. A similar result was reported in the paper \cite{ZhangDuchWai15}, which also treats the low smoothness case. 

\subsubsection{High smoothness}

We choose $\fo(x) = \frac{1}{2\pi}\sin(2\pi x)$, 
which corresponds to just one non-vanishing Fourier coefficient and by our criterion Corollary \ref{finalfourier} has $r=\infty$ . 
In view of our main  Corollary \ref{maintheo:total_averaged} this requires a regularization method with higher qualification; we take the {\it Gradient Descent} method (see Example \ref{landweber}).

The appearance of the term $2b\min\{1, r\}$ in our theoretical result \ref{maintheo:total_averaged} gives a predicted value 
$\alpha =0$ (and would imply that parallelization is strictly forbidden for infinite smoothness). More specifically, the left panel in Figure \ref{fig:three} shows the absence of any plateau for the reconstruction error as a function of $\alpha$. This corresponds to the right panel showing that no group of values of $\alpha$ performs roughly equivalently, meaning that we do not have any optimality guarantees.

Plotting different values of regularization in Figure  \ref{fig:four} we again identify overregularization as hopeless, while severe underregularization  exhibits a sharp minimum
in the reconstruction error. But its value at roughly $0.25$ is much less attractive compared to the case of low smoothness where the error is an order of magnitude less.


\begin{figure}[H]
\includegraphics[width=80mm, height=80mm]{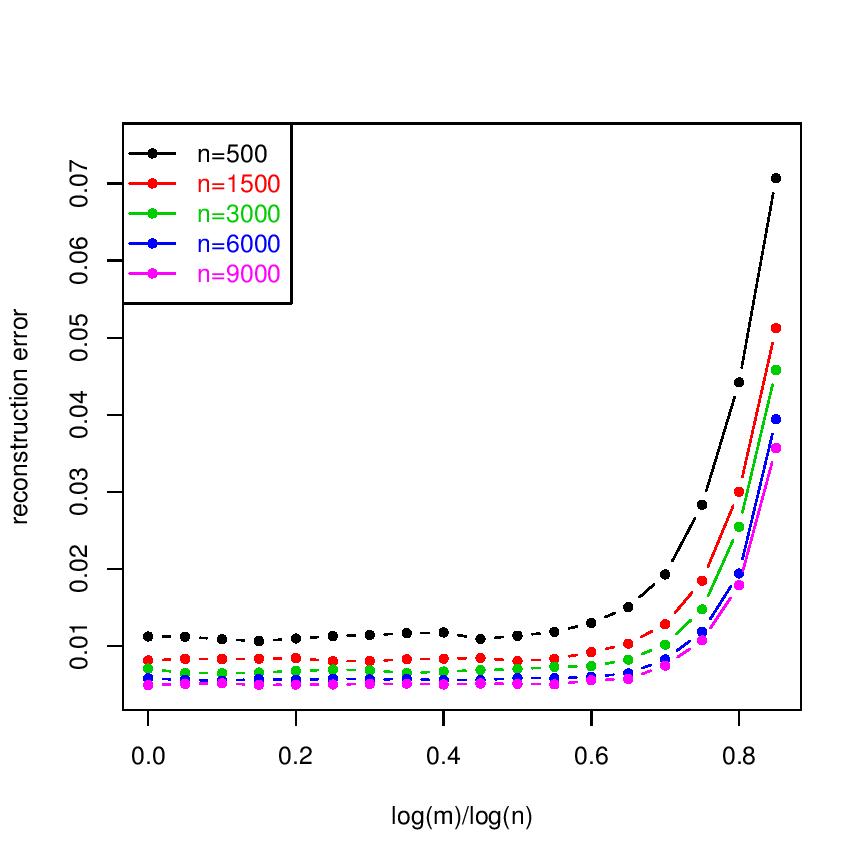}
\includegraphics[width=80mm, height=80mm]{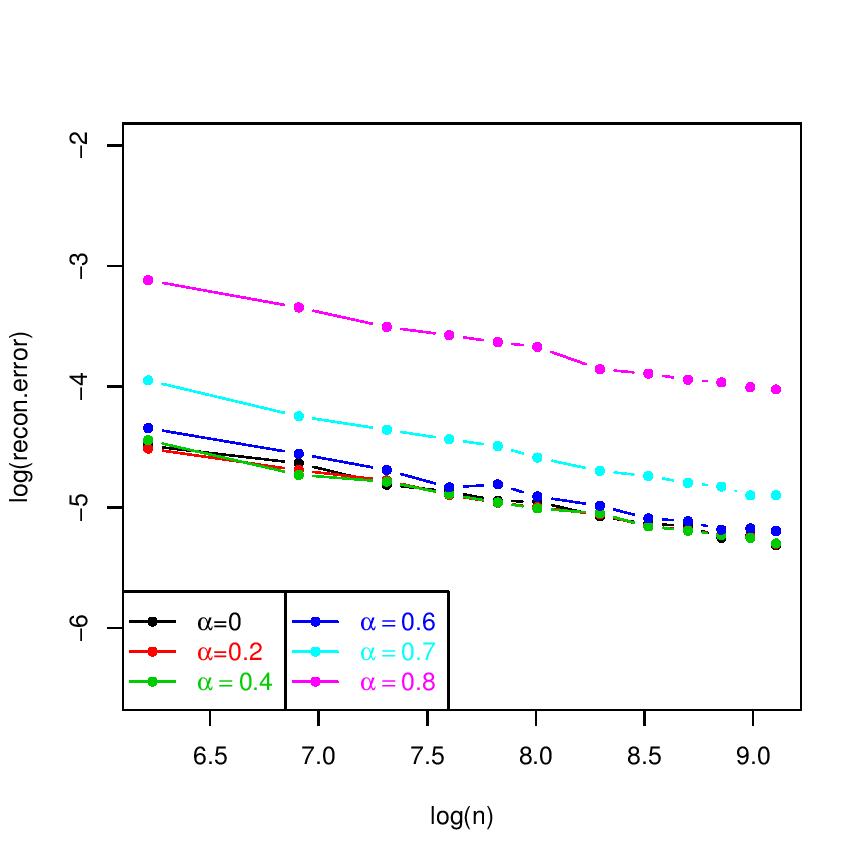}
\caption{{\small The reconstruction error  $||\bar f_{D}^{k_{oracle}} -\fo||_{\h}$ in the low smoothness case. 
Left plot: Reconstruction error curves for various (but fixed) sample sizes as a function of the number of partitions. 
Right plot: Reconstruction error curves for various (but fixed) numbers of partitions as a function of the sample size (on log-scale).} }
\label{fig:one}
\end{figure}

\begin{figure}[H]
\includegraphics[width=80mm, height=80mm]{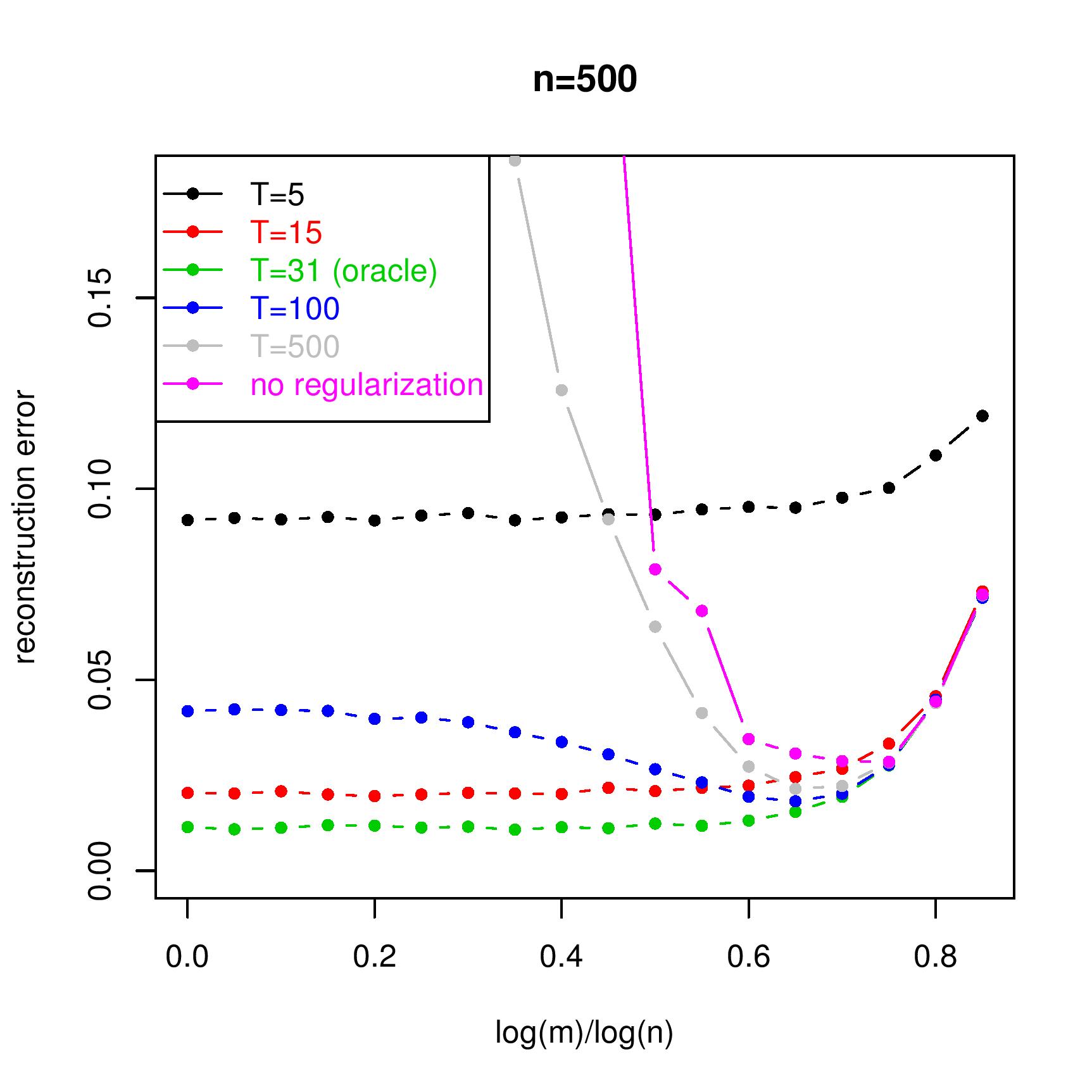}
\includegraphics[width=80mm, height=80mm]{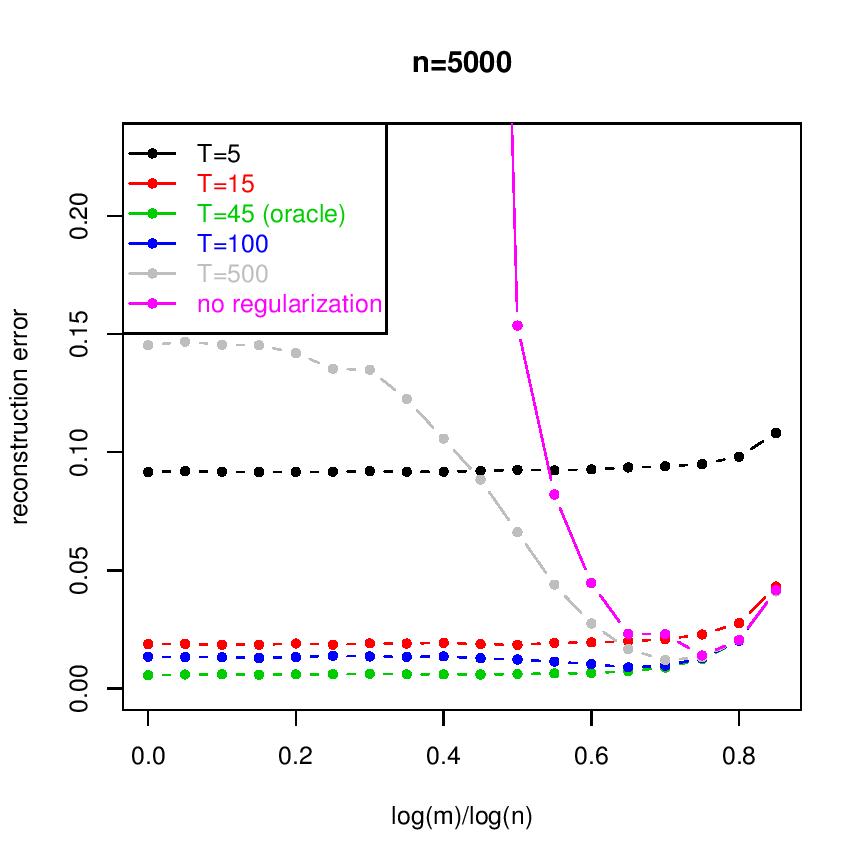}
\caption{{\small The reconstruction error  $||\bar f_{D}^{\lam} -\fo||_{\h}$ in the low smoothness case. 
Left plot: Error curves for different stopping times for  $n=500$ samples,   as a function of the number of partitions. 
Right plot: Error curves for different stopping times for  $n=5000$ samples,   as a function of the number of partitions.}} 
\label{fig:two}
\end{figure}

\begin{figure}[H]
\includegraphics[width=80mm, height=80mm]{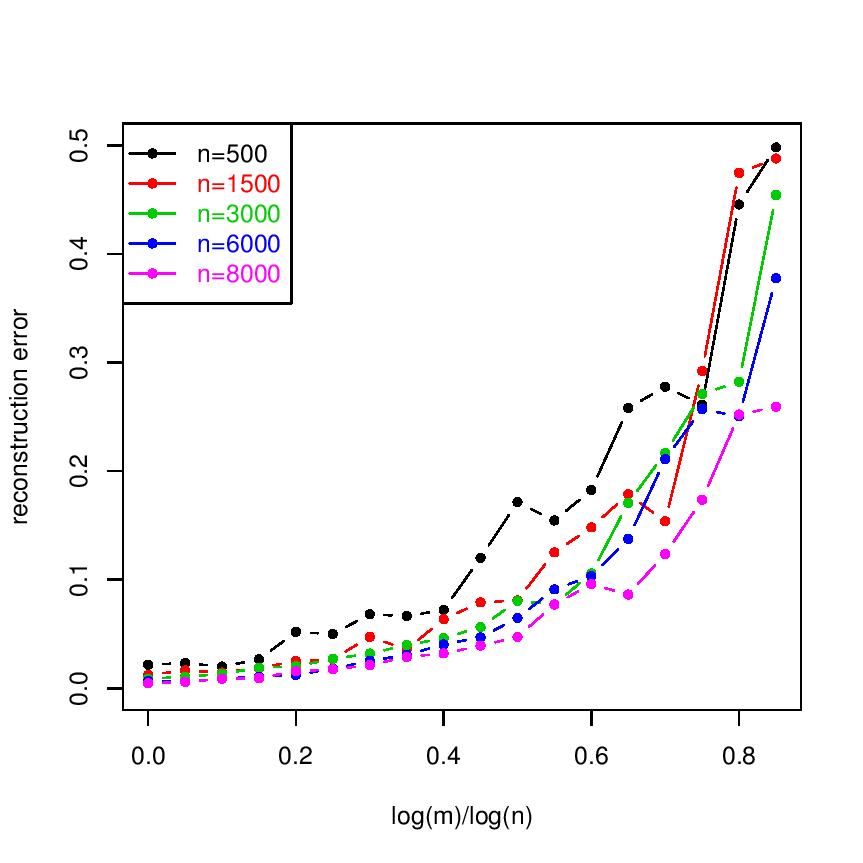}
\includegraphics[width=80mm, height=80mm]{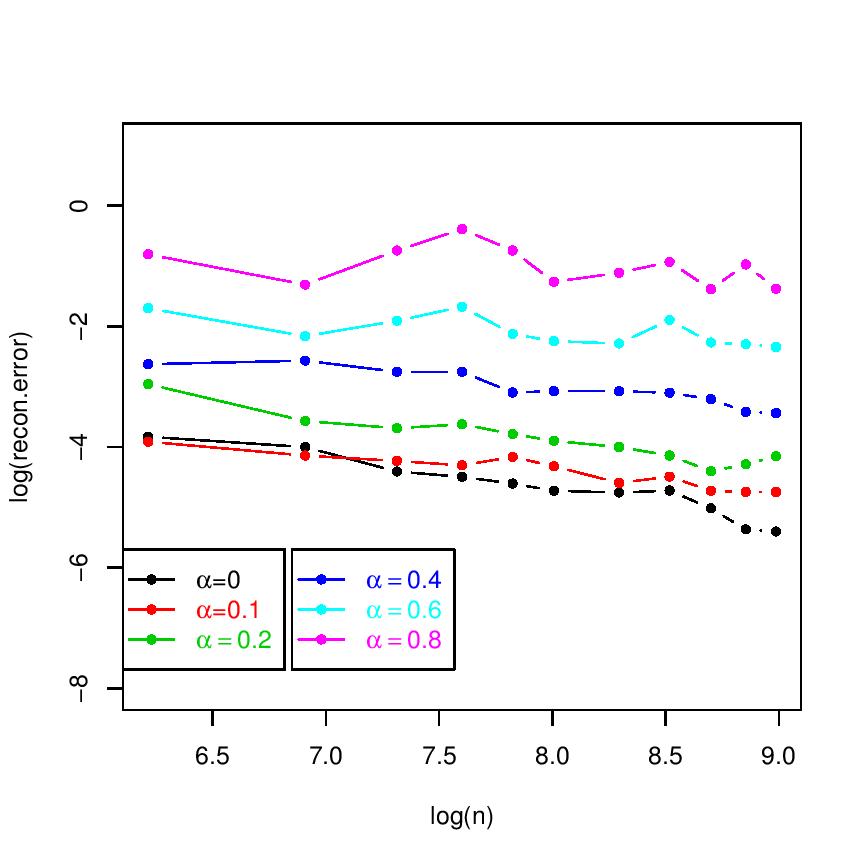}
\caption{{\small The reconstruction error  $||\bar f_{D}^{\lam_{oracle}} -\fo||_{\h}$ in the high smoothness case. 
Left plot: Reconstruction error curves for various (but fixed) sample sizes as a function of the number of partitions. 
Right plot: Reconstruction error curves for various (but fixed) numbers of partitions as a function of the sample size (on log-scale).} }
\label{fig:three}
\end{figure}

\begin{figure}[H]
\includegraphics[width=80mm, height=80mm]{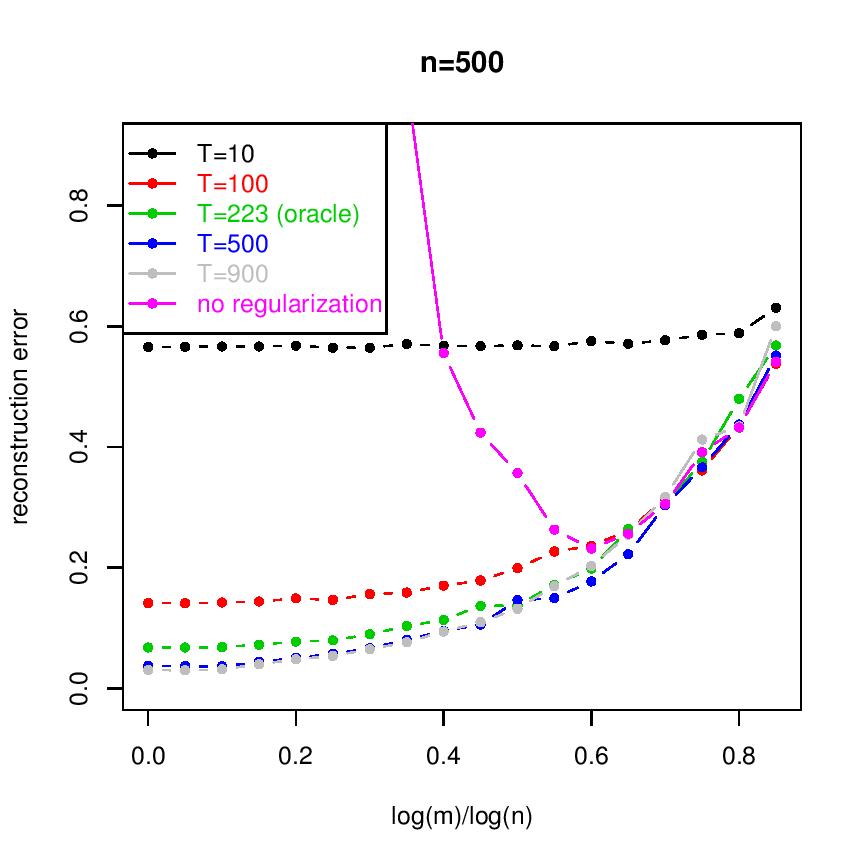}
\includegraphics[width=80mm, height=80mm]{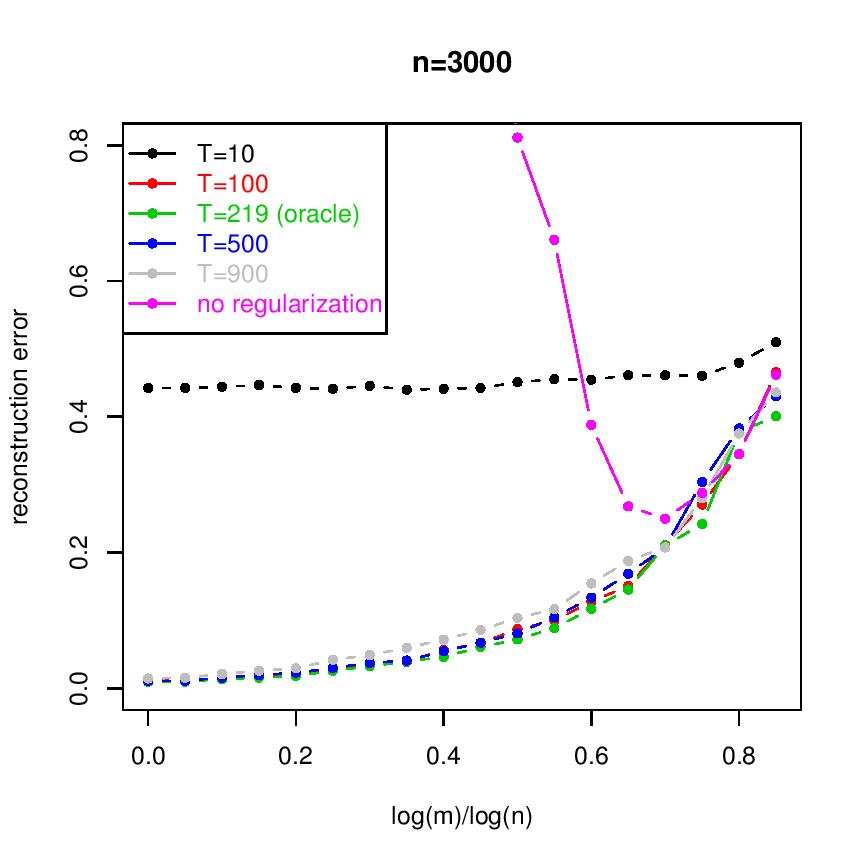}
\caption{{\small The reconstruction error  $||\bar f_{D}^{\lam} - \fo||$ in the high smoothness case. 
Left plot: Error curves for different stopping times for  $n=500$ samples,   as a function of the number of partitions. 
Right plot: Error curves for different stopping times for  $n=5000$ samples,   as a function of the number of partitions.}} 
\label{fig:four}
\end{figure}



\section{Discussion}
\label{discussion}

{\bf Minimax Optimality:}
We have shown that for a large class of spectral regularization methods the error of the distributed algorithm 
$|| \bar T^s(\bar f_D^{\lam_n}- \fo)||_{\h}$ satisfies the same upper bound as the error $\norm{ \bar T^s(f_D^{\lam_n}- \fo)}_{\h}$ for the single machine problem,  if the regularization parameter $\lam_n$ is chosen according to (\ref{eq:choicelam:averaged}), provided the number of subsamples grows sufficiently slowly with the sample size $n$\,. 
Since, by  \cite{BlaMuc16}, the rates for the latter are minimax optimal, our rates 
in Corollary \ref{maintheo:total_averaged} are minimax optimal also.
\\
\\
{\bf Comparison with other results:}
In \cite{ZhangDuchWai15} the authors derive Minimax-optimal rates in 3 cases: finite rank kernels, sub- Gaussian decay of eigenvalues of the kernel and polynomial decay, provided $m$  satisfies a certain upper bound, depending on the rate of decay of the eigenvalues under two crucial assumptions on the eigenfunctions of the integral operator associated to the kernel: For any $j \in \N$
\begin{equation}
\label{ass:eigen_1}
  \E [\phi_j(X)^{2k}] \leq \rho_k^{2k} \;,   
\end{equation}  
for some $k\geq 2$ and $\rho_k < \infty$ or even stronger, it is assumed  that the eigenfunctions are uniformly bounded, i.e.
\begin{equation}
\label{ass:eigen_2}
  \sup_{x \in \X} |\phi_j(x)| \leq \rho\;, 
\end{equation}  
or any $j \in \N$ and some $\rho < \infty$.  
We shall describe in more detail the case of polynomially decaying eigenvalues, which corresponds to our setting. Assuming eigenvalue decay $\mu_j \lesssim j^{-b}$ with $b > 1$, the authors choose a regularization parameter $\lam_n = n^{-\frac{b}{b+1}}$ and
\[    m \lesssim \paren{ \frac{n^{\frac{b(k-4)-k}{b+1}}}{\rho^{4k}\log^k(n)} }^{\frac{1}{k-2}} \;.\]
leading to an error in $L^2$- norm 
\[  \E[ ||\bar f_D^{\lam_n}- \fo ||^2_{L^2} ] \lesssim n^{-\frac{b}{b+1}}         \;, \]
being minimax optimal. 

For $k < 4$, this is not a useful bound, since $m \to 1$ as $n \to \infty$ in this case (for any sort of eigenvalue decay). On the other hand, if $k$ and $b$ might be taken arbitrarily large
- corresponding to almost bounded eigenfunctions and arbitrarily large polynomial decay of eigenvalues -  $m$  might be chosen proportional to $n^{1- \epsilon}$, for any $\epsilon >0$. As might be expected, replacing the $L^{2k}$ bound on the eigenfunctions by a bound in $L^\infty$, gives an upper bound on $m$ which simply is the limit for $k \to \infty$ in the bound given above, namely
\[     m \lesssim  \frac{n^{\frac{b-1}{b+1}}}{\rho^4  \log n}\; ,        \]
which for large $b$ behaves as above. Granted bounds on the eigenfunctions in $L^{2k}$ for (very) large $k$, this is a strong result. 
While the decay rate of the eigenvalues can be determined by the smoothness  of  $K$ (see, e.g., \cite{Men09} and references therein), it is a widely open question  which general  properties of the kernel imply estimates as in \eqref{ass:eigen_1} and \eqref{ass:eigen_2} on the eigenfunctions. 

The author in \cite{Zhou2002} even gives a counterexample and presents a $C^{\infty}$ Mercer kernel on $[0,1]$ where the eigenfunctions of the corresponding integral operator are {\it not} uniformly bounded. Thus, smoothness of the kernel is not a sufficient condition for \eqref{ass:eigen_2} to hold.

Moreover, we point out that the upper bound \eqref{ass:eigen_1} on the eigenfunctions (and thus the upper bound for $m$ in \cite{ZhangDuchWai15}) depends on the (unknown) marginal distribution $\nu$ (only the strongest assumption, a bound in sup-norm \eqref{ass:eigen_2}, does not depend on $\nu$). Concerning this point, our approach  
is "agnostic".

As already mentioned in the Introduction,  these bounds on the eigenfunctions have been eliminated in \cite{Zhou15}, for KRR, imposing polynomial decay of eigenvalues as above. This is very similar to our approach. 
As a general rule, our bounds on $m$ and the bounds  in \cite{Zhou15} are worse than the bounds in \cite{ZhangDuchWai15} for eigenfunctions in (or close to ) $L^\infty$, but in the complementary case where nothing is known on the eigenfunctions $m$ still can be chosen as an increasing function of $n$, namely $m=n^\alpha$.  
More precisely, choosing $\lam_n$ as in (\ref{eq:choicelam:averaged}), the authors in \cite{Zhou15} derive as an upper bound 
\[ m \lesssim n^{\alpha} \; , \qquad \alpha = \frac{2br}{2br + b + 1}  \;,    \]
with $r$ 
being the smoothness parameter arising in the source condition. We recall here that due to our assumption $q \geq r + s$, 
the smoothness parameter $r$ is restricted to the interval $(0, \frac{1}{2}]$ for KRR ($q=1$) and $L^2$ risk ($s=\frac{1}{2}$).

Our results (which hold for a general class of spectral regularization methods) are in some ways comparable to  \cite{Zhou15}. 
Specialized to KRR, our estimates for the exponent $\alpha$  in $m= O(n^\alpha)$ coincide with the result given in \cite{Zhou15}\;. 
Furthermore we emphasize that \cite{ZhangDuchWai15} and \cite{Zhou15} estimate the DL-error only for $s=1/2$ in our notation (corresponding to  $L^2(\nu)-$ norm), while our result holds for all values of $ s \in [0,1/2]$ which smoothly interpolates between $L^2(\nu)-$ norm and RKHS$-$ norm and, in addition, for all values of $p \in [1,\infty)$. Thus, our results also apply to the case of non-parametric inverse regression, where one is particularly interested in the reconstruction error (i.e. $\h$- norm), see e.g. \cite{BlaMuc16}. 
Additionally, we precisely analyze the dependence of the noise variance $\sigma^2$  and the complexity radius $R$ in the source condition.


Concerning general strategy, while \cite{Zhou15} uses a novel second order decomposition in an essential way, our approach is more classical. 
We clearly distinguish between estimating the approximation error and the sample error.  The bias using a subsample should be of the
same order as when using the whole sample, whereas the estimation error is higher on each subsample, but gets
reduced by averaging by writing the variance as a sum of i.i.d random variables (which allows to use Rosenthal's inequality). 

Finally, we want to mention the recent works \cite{Lin17} and \cite{GuLiZh17}, which were worked out indepently from our work. The authors in \cite{GuLiZh17} also treat general 
spectral regularization methods (going beyond kernel ridge) and obtain essentially the same results, but with error bounds only in $L^2$- norm, excluding inverse learning problems. 
In \cite{Lin17}, the authors investigate distributed learning on the example of Gradient Descent algorithms, which have infinite qualification and allow larger smoothness of the regression function. They are able to improve the upper bound for the number of local machines to
\[  m \lesssim \frac{n^{\alpha}}{\log^5(n)+1} \;, \qquad \alpha < \frac{br}{2br+b+1}   \]
which is larger in case $r>2$. In the intermediate case $1<r<2$, our bound in \eqref{cond:m_alpha_final} is still better. 
An interesting feature is the fact that it is possible to allow more local machines by using additional unlabeled data. This indicates that finding the upper bound for the number of machines in the high smoothness regime is still an open problem.  
\\
\\
{\bf Number of Subsamples:}
We follow the line of reasoning in earlier work on distributed learning insofar as we only prove {\em sufficient conditions} on the cardinality 
$m=n^\alpha$ of subsamples compatible with minimax optimal rates of convergence. On the complementary problem of proving {\em necessity}, analytical results are unknown to the best of our knowledge. However, our numerical results seem to indicate that the exponent $\alpha$ might actually be taken larger than we have proved so far in the low smoothness regime. 
\\
\\
{\bf Adaptivity:}
It is clear from the theoretical results that both the regularization parameter $\lambda$ and the allowed cardinality of subsamples $m$ depend on the parameters  
$r$ and $b$, which in general are unknown. Thus, an adaptive approach to both parameters $b$ and $r$ for choosing $\lambda$ and $m$  is of interest.  
To the best of our knowledge, there are yet no rigorous results on adaptivity in this more general sense. Progress in this field may well be crucial in finally assessing the relative merits of the distributed learning approach as compared with alternative strategies to effectively deal with large data sets.

We sketch an alternative naive approach to adaptivity, based on hold-out in the direct case, where we consider each $f \in \hhh_K$ also as a function in $L^2(\X, \nu)$. 
We split the data $\z \in (\X \times \Y)^n$ into a training and  validation part $\z=(\z^t,\z^v)$ of cardinality $m_t, m_v$. 
We further subdivide $\z^t$ into $m_k$ subsamples, roughly of size $m_t/m_k$, where $m_k \leq m_t, \, k=1,2, \ldots $  is some strictly decreasing sequence. For each $k$ and each subsample 
$\z_j,$  $ 1 \leq j \leq m_k$, 
we define the estimators $\hat f^\lam_{\z_j}$ as in \eqref{def:subset_estimator} and their average 
\begin{equation}  \label{estav}
\bar f^\lam_{k,\z^t} := \frac {1}{m_k} \sum_{j=1}^{m_k} \hat f^\lam_{\z_j}\;.
\end{equation} 
Here, $\lam$ varies in some sufficiently fine lattice $\Lambda$. Then evaluation on $\z^v$ gives the associated empirical $L^2- $ error
\begin{equation}  \label{emperr}
\Err^\lam_k(\z^v):= \frac {1}{m_v} \sum_{i=1}^{ m_v} |y^{v}_i - \bar f^\lam_{k,\z^t}(x^v_i)|^2\;,  \quad \z^v=(\y^v,\x^v)\;, \quad \y^v=(y^v_1, \ldots ,y^v_{m_v})\;,
\end{equation}
leading us to define
\begin{equation}  \label{errk}
\hat \lam_k:=  \mbox{Argmin}_{\lam \in \Lambda}  \Err^\lam_k(\z^v)\;,  \qquad
\Err(k):= \Err_k^{\hat \lam_k}(\z^v). 
\end{equation}
Then, an appropriate stopping criterion for 
$k$ might be to stop at
\begin{equation}  \label{discr}
k^*:= \min \{k \geq 3\; : \;  \Delta(k) \leq \delta \inf_{2 \leq j < k} \Delta(j) \}\;, \qquad
\Delta(j):= |\Err(j) - \Err(j-1)|\;,
\end{equation}
for some $\delta < 1$ (which might require tuning). The corresponding regularization parameter is $\hat \lam = \hat \lam_{k^*}$, given by \eqref{errk}. At least intuitively, it is then reasonable to define a purely data driven estimator as
\begin{equation}  \label{adapt}
\fest:= \bar f^{\hat \lam}_{k^*,\z^t} \;.
\end{equation}
 Note that the training data $\z^t$ enter the definition of $\fest$ via the explicit formula \eqref{estav} encoding our kernel based approach, while $\z^v$ serves to determine
 $(k^*,\hat \lam^*)$ via minimization of the empirical $L^2-$ error and some form of the discrepancy principle, which tells one to stop where $\Err(j)$ does not appreciably improve anymore.
It is open if such a procedure achieves optimal rates, and we have to leave this for future research.



\section{Proofs}

For ease of reading we make use of the following conventions: 
\begin{itemize}
 \item we are interested in a precise dependence of multiplicative constants on the parameters $\sigma, M, R,\eta$, $m, n$ and $p$ 
 \item the dependence of multiplicative constants on various other parameters, including the kernel parameter $\kappa$, the norm 
 parameter $s \in [0, \frac{1}{2}]$, the parameters arising from the regularization method, $b>1$, $\beta>0$, $r>0$,  etc.  will
   (generally) be omitted and simply indicated by the symbol $\blacktriangle$
 \item the value of $C_{\blacktriangle}$ might change from line to line  
 \item the expression ``for $n$ sufficiently large'' means that the statement holds for
   $n\geq n_0$\,, with $n_0$ potentially depending on all model parameters
   (including $\sigma, M$ and $R$), but not on $\eta$\,.
 \end{itemize}

\subsection{Preliminaries}

For proving our error bounds, we recall some results (without proof) from  \cite{optimalratesRLS}.   
We introduce the {\it effective dimension} $\NN (\lam)$, being a measure for the complexity of $\h$ with respect to the marginal distribution $\nu$: 
For $\lambda \in (0,1]$ we set
\begin{equation}
\label{eq:effectivedim}
\NN (\lam) = \tr(\;(\bar T+ \lam)^{-1} \bar T\;) \;.
\end{equation}  
Since the operator $T$ is trace-class,
$\NN(\lam)< \infty$. Moreover, $\NN(\lam)$ satisfies
    \[  \frac{1}{2} \leq \NN(\lam)\leq  \frac{\beta b}{b-1} (\kappa^2\lam) ^{-\frac{1}{b}}\; , 
    \]
provided the marginal distribution $\nu$ of $\X$ belongs to $\priorle(b,\beta)$ with $b>1$ and $\beta>0$  
(see \cite{optimalratesRLS}, Proposition 3).

\begin{prop}[\cite{GuLiZh17}, Proposition 1]
\label{prop:Guo}
Let $x_1, ..., x_n$ be an iid sample, drawn according to $\nu$ on $\X$. 
Define
\begin{equation}
\label{def:blam}
  \B_n(\lam) :=  \left[1 + \left( \frac{2}{n\lam} + \sqrt{\frac{\NN(\lam)}{n\lam}}\right)^2 \right] 
\end{equation}  
For any $\lam >0$,  $\eta \in (0,1]$, with probability at least $1-\eta$ one has 
\begin{equation}
\label{Guo:estimate}
 \norm{ (\bar T_{\x}+\lam)^{-1}(\bar T+\lam) } \leq 
     8\log^2(2\etainv)   \B_n(\lam) \;.
\end{equation}
\end{prop}

\begin{cor}
\label{cor:Peter}
Let $\eta \in (0,1)$. For  $n \in \N$ let $\tilde \lam_n$ be implicitly defined as the unique solution of $\NN(\tilde \lam_n) =  n \tilde \lam_n$. Then  for any $\tilde \lam_n \leq \lam \leq 1$ one has
\begin{equation*}
  \B_n(\lam)  \leq    26\;.
\end{equation*}
In particular,
\begin{equation*}
\norm{ (\bar T_{\x}+\lam)^{-1}(\bar T+\lam)} \leq 208 \log^2(2\etainv) \;,
\end{equation*}
with probability at least $1-\eta$.
\end{cor}

\begin{proof}[Proof of Corollary \ref{cor:Peter}]
Let $\tilde \lam_n$ be defined via $\NN(\tilde \lam_n) =  n \tilde \lam_n$. Since $\NN(\lam)/\lam$ is decreasing, we have for any $\lam \geq \tilde \lam_n$
\[  \sqrt{\frac{\NN(\lam)}{n\lam}} \leq \sqrt{\frac{\NN(\tilde \lam_n)}{n\tilde \lam_n}} = 1 \;.   \]
Since the effective dimesion is lower bounded by $\frac{1}{2}$, by the inequality above
\[   1 \geq   \sqrt{\frac{\NN(\lam)}{n\lam}} \geq \frac{1}{\sqrt 2 n\lam} \quad \Longrightarrow \quad \frac{1}{n\lam} \leq \sqrt 2 < 2  \]
for any $\lam \geq \tilde \lam_n$. Inserting these bounds into  \ref{Guo:estimate} and noticing that $1 \leq 2\log(2\etainv)$ for any $\eta \in (0,1)$ 
leads to the conclusion. 
\end{proof}


\begin{cor}
\label{lem:for_all}
If $\lam_n$ is defined by \eqref{eq:choicelam:averaged} and if 
\[  m_n\leq n^{\alpha}\;, \qquad \alpha < \frac{2br}{2br+b+1} \;, \]
one has 
\[  \B_{\frac{n}{m_n}}(\lam_n) \leq 2 \;, \]
provided $n$ is sufficiently large. 
\end{cor}

\begin{proof}[Proof of Lemma \ref{lem:for_all}]
Recall that $\NN(\lam_n) \leq C_b \lam_n^{-\frac{1}{b}}$ and $\sigma \sqrt{ \frac{\lam_n^{-\frac{1}{b}}}{n \lam_n} } = R \lam_n^r$.  
Using the definition of $\lam_n$ in \eqref{eq:choicelam:averaged} yields
\[ \frac{2m_n}{n\lam} = o\paren{ \sqrt m_n \lam_n^r }\;, \]
provided
\[  m_n\leq n^{\alpha}\;, \qquad \alpha < \frac{2(br+1)}{2br+b+1} \;. \]
Finally,  $\sqrt m_n \lam_n^r = o(1)$ if 
\[  m_n\leq n^{\alpha}\;, \qquad \alpha < \frac{2br}{2br+b+1} \;. \]
\end{proof}


We shortly illustrate how Corollary \ref{cor:Peter} and Proposition \ref{prop:Guo} will be used. Let $u \in [0, 1]$, 
$\tilde \lam_n \leq \lam$ as above and $f \in \h$. We have 
\begin{align}
\norm{\bar T^uf}_{\h} &= \norm  {\bar T^u (\bar T+\lam)^{-u}(\bar T+\lam)^u(\bar T_{\x}+\lam)^{-u} (T_{\x}+\lam)^uf  }_{\h} \nonumber  \\
&\leq   \norm{\bar T^u (\bar T+\lam)^{-u}} \norm{(\bar T+\lam)^u(\bar T_{\x}+\lam)^{-u}}    \norm{ (\bar T_{\x}+\lam)^uf  }_{\h} \nonumber  \\
&\leq  8\log^{2u}(2\etainv) \B_{n}(\lam)^u    \norm{ (\bar T_{\x}+\lam)^uf  }_{\h} \label{need} \;, 
\end{align}
with probability at least $1-\eta$, for any $\eta \in (0,1)$. 
In particular, for any $\tilde \lam_n \leq \lam$ (with $\tilde \lam_n$ as in Corollary \ref{cor:Peter})
\begin{equation}
\label{eq:MOD}
\norm{\bar T^uf}_{\h}  \leq    208^u\log^{2u}(2\etainv) \norm{ ( \bar T_{\x}+\lam)^uf  }_{\h}\;, 
\end{equation}
with probability at least $1-\eta$.
\\
\\
In the following, we constantly use \eqref{eq:MOD}.


\subsection{{\bf Approximation Error Bound}}

Recall that $\nu$ denotes the input sampling distribution and ${\cal P}$ the set of all probability distributions on the input space $\X$.


\begin{lem}
\label{lem:approx1}
Let $\nu \in {\cal P}$, $v \in \R$ and let $\x \in \X^{\frac{n}{m}}$ be an iid sample, drawn according to $\nu$.   
Assume the regularization $(g_{\lam})_{\lam}$ has qualification $q \geq v+1+s$. 
Then with  probability at least $1-\eta$
\begin{equation*}
\norm{ \bar T^s r_{\lam}(\bar T_{\x})\bar T_{\x}^v(\bar T - \bar T_{\x}) }_{\h} \leq C_{\blacktriangle}\log^{4}(4\etainv) \lam^{s+v+1}\B^{s+1}_{\frac{n}{m}}(\lam) \paren{ \frac{2m}{n\lam} + \sqrt{\frac{m\NN(\lam)}{n\lam}}}
\end{equation*}
for some $C_{\blacktriangle} < \infty $.
\end{lem}

\begin{proof}[Proof of Lemma \ref{lem:approx1}]
From \eqref{need} and from Proposition \ref{Geta2}, since $q \geq s+v+1$, one has 
\begin{align*}
\norm{ \bar T^s r_{\lam}(\bar T_{\x})\bar T_{\x}^v(\bar T - \bar T_{\x}) }_{\h} &\leq C_{\blacktriangle}\log^{2(s+1)}(4\etainv)\B^{s+1}_{\frac{n}{m}}(\lam) \\
   & \qquad \norm{(\bar T_{\x} + \lam)^{s} r_{\lam}(\bar T_{\x})\bar T_{\x}^{v}(\bar T_{\x} + \lam)   } 
     \norm{(\bar T + \lam)^{-1}(\bar T - \bar T_{\x})}  \\
     &\leq C_{\blacktriangle}\log^{4}(4\etainv) \lam^{s+v+1}\B^{s+1}_{\frac{n}{m}}(\lam) \paren{ \frac{2m}{n\lam} + \sqrt{\frac{m\NN(\lam)}{n\lam}}}\;,
\end{align*}  
for any $\lam \in (0,1]$, $\eta \in (0,1]$, with probability at least $1-\eta$. We also used that $s \leq \frac{1}{2}$. 
\end{proof}


\begin{lem}
\label{lem:approx2}
Let $\nu \in {\cal P}$, $v \in \R$ and let $\x \in \X^{\frac{n}{m}}$ be an iid sample, drawn according to $\nu$. 
Assume the regularization $(g_{\lam})_{\lam}$ has qualification $q \geq v+s$. 
Then for any $\lam \in (0,1]$, $\eta \in (0,1]$, with probability at least $1-\eta$
\begin{equation*}
\norm{ \bar T^s r_{\lam}(\bar T_{\x})\bar T_{\x}^v } \leq C_{\blacktriangle}\log^{2s}(2\etainv)\B^s_{\frac{n}{m}}(\lam)\lam^{s+v}  \;,
\end{equation*}
for some $C_{\blacktriangle} < \infty $.
\end{lem}

\begin{proof}[Proof of Lemma \ref{lem:approx2}]
Using  \eqref{need}, since $q \geq v+s$
\begin{align*}
\norm{ \bar T^s r_{\lam}(\bar T_{\x})\bar T_{\x}^v } &\leq  C_{\blacktriangle}\log^{2s}(2\etainv)\B^s_{\frac{n}{m}}(\lam)
                     \norm{ (\bar T_{\x} + \lam)^s r_{\lam}(\bar T_{\x})\bar T_{\x}^v }  \\
                  &\leq  C_{\blacktriangle}\log^{2s}(2\etainv)\B^s_{\frac{n}{m}}(\lam)\lam^{s+v} \;,   
\end{align*}
with probability at least $1-\eta$. 
\end{proof}


\begin{prop}[Expectation of Approximation Error]
\label{prop:expec_Approx}
Let $\fo \in \Omega_{\nu}(r, R)$, $\lam \in (0,1]$ and let $\B_{\frac{n}{m}}(\lam)$ be defined in \eqref{def:blam}. Assume the regularization has qualification $q \geq r+s$. For any $p \geq 1$ one has:
\begin{enumerate}
\item If $r \leq 1$, then
\[ \E_{\rho^{\otimes n}}
     \Big[\big\|\bar T^s( \fo - \tilde f_{D}^{\lam})\big\|^p_{\h}\Big]^{\frac{1}{p}} 
        \leq   C_p R \; \lam^{s+r}\; \B^{s+r}_{\frac{n}{m}}(\lam) \;. \] 
\item If $r >1$, then 
\[ \E_{\rho^{\otimes n}}
     \Big[\big\|\bar T^s( \fo - \tilde f_{D}^{\lam})\big\|^p_{\h}\Big]^{\frac{1}{p}} \leq   C_{p}R \lam^s \B^{s+1}_{\frac{n}{m}}(\lam)\paren{\lam^{r} + \lam  \paren{ \frac{2m}{n\lam} + \sqrt{\frac{m\NN(\lam)}{n\lam}}}  
         } \;.\] 
\end{enumerate}
In 1. and 2. the constant $ C_{p}$ does not depend on $(\sigma, M,R)\in \R^3_+$.
\end{prop}

\begin{proof}[Proof of Proposition \ref{prop:expec_Approx}]
Since $\fo \in \Omega_{\nu}(r, R)$
\begin{align}
\label{eq:deco5}
 \E_{\rho^{\otimes n}}
     \Big[\big\|\bar T^s( \fo - \tilde f_{D}^{\lam})\big\|^p_{\h}\Big]^{\frac{1}{p}} &= 
   \E_{\rho^{\otimes n}}\Big[ \big\|  \frac{1}{m}\sum_{j=1}^m \bar T^sr_{\lam}(\bar T_{\x_j})\fo  \big\|_{\h} ^p  \Big]^{\frac{1}{p}}   \nonumber   \\
   &\leq   \frac{1}{m}   \sum_{j=1}^m        \E_{\rho^{\otimes n}}\Big[ \big\|   \bar T^sr_{\lam}(\bar T_{\x_j})\fo    \big\|_{\h} ^p  \Big]^{\frac{1}{p}}        \nonumber \\
 &\leq \frac{R}{m}  \sum_{j=1}^m \E_{\rho^{\otimes n}}\Big[ \big\|\bar T^sr_{\lam}(\bar T_{\x_j})\bar T^{r}   \big\|_{\h}^p  \Big]^{\frac{1}{p}} \;.
\end{align}
The first inequality is just the triangle inequality for the $p$- norm $||f||_p = \E[||f||^p_{\h} ]^{\frac{1}{p}}$.  
We bound the expectation for each separate subsample of size $\frac{n}{m}$ by first deriving a probabilistic estimate and then we integrate.

Consider first the case where $r \leq 1$. Using \eqref{need} and Cordes Inequality Proposition \ref{prop:Bath-ineq}\;, one has for any $j = 1,...,m$
\begin{align*}
 \norm{\bar T^sr_{\lam}(\bar T_{\x_j})\bar T^{r}} &\leq C_{\blacktriangle}\log^{2(s+r)}(4\etainv) \B^{s+r}_{\frac{n}{m}}(\lam) 
   \norm{(\bar T_{\x_j}+\lam)^s r_{\lam}(\bar T_{\x_j})(\bar T_{\x_j}+\lam)^{r}}  \\
  &\leq C_{\blacktriangle}\log^{3}(4\etainv) \lam^{s+r}\B^{s+r}_{\frac{n}{m}}(\lam)\;,  
\end{align*}
with probability at least $1-\eta$ and where $\B^{s+r}_{\frac{n}{m}}(\lam)$ is defined in \eqref{def:blam}. Recall that the regularization has qualification $q\geq r+s$. 
By integration one has
\begin{align*}
\E_{\rho^{\otimes n}}\Big[ \big\|\bar T^sr_{\lam}(\bar T_{\x_j})\bar T^{r}   \big\| ^p  \Big]^{\frac{1}{p}} &\leq C_{\blacktriangle, p} \;
    \lam^{s+r}\; \B^{s+r}_{\frac{n}{m}}(\lam)\;,
\end{align*}
for some $C_{\blacktriangle, p} < \infty$, not depending on $\sigma,M,R$. 
Finally, from \eqref{eq:deco5} 
\[ \E_{\rho^{\otimes n}}
     \Big[\big\|\bar T^s( \fo - \tilde f_{D}^{\lam})\big\|^p_{\h}\Big]^{\frac{1}{p}} 
        \leq   C_{\blacktriangle, p} R \; \lam^{s+r}\; \B^{s+r}_{\frac{n}{m}}(\lam) \;. \]

In the case where $r \geq 1$, we write $r=k+u$, with $k=\lfloor r \rfloor$ and $u=r-k<1 $. We shall use the decomposition
\begin{equation}
\label{eq:decomp}
\bar T^k = \sum_{l=0}^{k-1}\bar T_{\x}^l(\bar T -\bar T_{\x}) \bar T^{k-(l+1)}+ \bar T_{\x}^{k} \;.
\end{equation}
We proceed by bounding \eqref{eq:deco5} according to decomposition \eqref{eq:decomp}\;. 
For any $j = 1,...m$, one has 
\begin{align}
\label{eq:deco3}
 \E_{\rho^{\otimes n}}\Big[ \big\| \bar T^sr_{\lam}(\bar T_{\x_j})\bar T^{k+u}   \big\| ^p  \Big]^{\frac{1}{p}} &\leq 
      \sum_{l=0}^{k-1} \E_{\rho^{\otimes n}}\Big[ \big\| \bar  T^sr_{\lam}(\bar T_{\x_j})\bar T_{\x_j}^l(\bar T - \bar T_{\x_j})\bar T^{k-(l+1)+u}  \big\| ^p  \Big]^{\frac{1}{p}}  \nonumber \\ 
      & \qquad + \E_{\rho^{\otimes n}}\Big[ \big\|\bar  T^sr_{\lam}(\bar T_{\x_j})\bar T_{\x_j}^{k}\bar T^u   \big\| ^p  \Big]^{\frac{1}{p}}  \nonumber \\
  &\leq \sum_{l=0}^{k-1} \E_{\rho^{\otimes n}}\Big[ \big\| \bar  T^sr_{\lam}(\bar T_{\x_j})\bar T_{\x_j}^l(\bar T - \bar T_{\x_j})  \big\| ^p  \Big]^{\frac{1}{p}}     \nonumber \\ 
& \qquad   + \E_{\rho^{\otimes n}}\Big[ \big\|\bar  T^sr_{\lam}(\bar T_{\x_j})\bar T_{\x_j}^{k}\bar T^u   \big\| ^p  \Big]^{\frac{1}{p}} \;.
\end{align} 
Here we use that $\bar T^{k-(l+1)+u}$ is bounded by $1$. 
By Lemma \ref{lem:approx2} and by \eqref{need}, with probability at least $1-\eta$
\begin{align*}
\norm{\bar  T^sr_{\lam}(\bar T_{\x_j})\bar T_{\x_j}^{k}\bar T^u} \leq   C_{\blacktriangle}\log^{2(s+u)}(2\etainv)\B^{s+u}_{\frac{n}{m}}(\lam)\lam^{s+r} 
\end{align*}
and thus integration yields
\begin{equation}
\label{eq:another1}
\E_{\rho^{\otimes n}}\Big[ \big\|\bar  T^sr_{\lam}(\bar T_{\x_j})\bar T_{\x_j}^{r} \bar T^u  \big\| ^p  \Big]^{\frac{1}{p}} \leq 
  C_{\blacktriangle, p}\B^{s+u}_{\frac{n}{m}}(\lam)\lam^{s+r} \;.
\end{equation}

For estimating the first term in \eqref{eq:deco3} we may use Lemma \ref{lem:approx1}. 
For any $l=0,...,k-1$, $j=1,...,m$ with probability at least $1-\eta$
\begin{small}
\begin{align*}
\norm{ \bar  T^sr_{\lam}(\bar T_{\x_j})\bar T_{\x_j}^l(\bar T - \bar T_{\x_j})   } &\leq 
C_{\blacktriangle}\log^{4}(8\etainv) \lam^{s+l+1}\B^{s+1}_{\frac{n}{m}}(\lam) \paren{ \frac{2m}{n\lam} + \sqrt{\frac{m\NN(\lam)}{n\lam}}} \;.
\end{align*}
\end{small}
Again by integration, since $\lam^l \leq 1$ for any $l =0, ..., k-1$, one has
\begin{small}
\begin{equation}
\label{eq:another2}
\sum_{l=0}^{k-1} \E_{\rho^{\otimes n}}\Big[ \big\| \bar  T^sr_{\lam}(\bar T_{\x_j})\bar T_{\x_j}^l(\bar T - \bar T_{\x_j})  \big\| ^p  \Big]^{\frac{1}{p}} \leq C_{\blacktriangle, p}
       \lfloor r \rfloor  \lam^{s+1} \B^{s+1}_{\frac{n}{m}}(\lam) \paren{ \frac{2m}{n\lam} + \sqrt{\frac{m\NN(\lam)}{n\lam}}} \;.
\end{equation}
\end{small}
Finally, combining \eqref{eq:another1} and \eqref{eq:another2} with \eqref{eq:deco5}  gives in the case where $r >1$ 
\[ \E_{\rho^{\otimes n}}
     \Big[\big\|\bar T^s( \fo - \tilde f_{D}^{\lam})\big\|^p_{\h}\Big]^{\frac{1}{p}} \leq   C_{\blacktriangle}\lam^s\B^{s+1}_{\frac{n}{m}}(\lam) \paren{ \lam^{r} + \lam  \paren{ \frac{2m}{n\lam} + \sqrt{\frac{m\NN(\lam)}{n\lam}}} 
        } \;.\]
The rest of the proof follows from \eqref{eq:deco3}.      
     
\end{proof}

\begin{proof}[Proof of Theorem \ref{prop:approx_error}]
Let $\lam_n$ defined by \eqref{eq:choicelam:averaged}.  According to Lemma \ref{lem:for_all}, we have $\B_{\frac{n}{m_n}}(\lam_n) \leq 2$ provided 
$\alpha < \frac{2br}{2br +b+1}$. We immediately obtain from the first part of Proposition  \ref{prop:expec_Approx} in the case where $r \leq 1$
\[ \E_{\rho^{\otimes n}}
     \Big[\big\|\bar T^s( \fo - \tilde f_{D}^{\lam_n})\big\|^p_{\h}\Big]^{\frac{1}{p}} 
        \leq   C_{\blacktriangle, p} R \; \lam_n^{s+r} = C_{\blacktriangle, p} \; a_n \;.\]

We turn to the case where $r > 1$. We apply the second part of Proposition  \ref{prop:expec_Approx}. By Corollary \ref{lem:for_all} we have 
\begin{align*}
 \E_{\rho^{\otimes n}}
     \Big[\big\|\bar T^s( \fo - \tilde f_{D}^{\lam_n})\big\|^p_{\h}\Big]^{\frac{1}{p}} &\leq   
C_{p}R \lam_n^s \B^{s+1}_{\frac{n}{m_n}}(\lam_n)\paren{\lam_n^{r} + \lam_n  \paren{ \frac{2m_n}{n\lam_n} + \sqrt{\frac{m_n\NN(\lam_n)}{n\lam_n}}} } \\
     &\leq C_{\blacktriangle, p} \;R \lam_n^{s} \; \paren{\lam_n^{r} + \lam_n  \paren{ \frac{2m_n}{n\lam_n} + \sqrt{\frac{m_n\NN(\lam_n)}{n\lam_n}}} } \;,
\end{align*}      
where we used that $\NN(\lam_n)\leq C_b \lam_n^{-1/b}$ and the definition of $\lam_n$. Observe that 
\[ \frac{2m_n}{n\lam_n} = o\paren{ \sqrt m_n \lam_n^r }\;, \]
provided
\[  m_n\leq n^{\alpha}\;, \qquad \alpha < \frac{2(br+1)}{2br+b+1} \;. \]
Furthermore, for $n$ sufficiently large, $\frac{R}{\sigma}\sqrt{m_n}\lam_n \leq 1$, provided that 
\[ \alpha < \frac{2b}{2br+b+1} \;.\]
As a result, for any $1\leq p$
\[
   \limsup_{n \rightarrow \infty} \sup_{\rho \in \mathcal{M}_{\sigma,M,R}} \frac{   \E_{\rho^{\otimes n}}
     \Big[\big\| \bar T^s( \fo -\tilde f^{\lam_n}_D  )   \big\|^p_{\h}\Big]^{\frac{1}{p}}}{a_{n}} \leq C_{\blacktriangle, p} \,,
\]
for some $C_{\blacktriangle, p} < \infty$, not depending on $\sigma,M,R$. 
\end{proof}


\subsection{{\bf Sample Error Bound}}

The main idea for deriving an upper bound for the sample error is to identify it as a sum of unbiased Hilbert space- valued i.i.d. variables and then to 
apply a suitable version of Rosenthal's inequality.

Given $\lam \in (0,1]$, we define the random variable $\xi_{\lam}:(\X \times \R)^{\frac{n}{m}} \longrightarrow \h$ by 
\[ \xi_{\lam}(\x, \y):= \bar T^s g_{\lam}(\bar T_{\x})(\bar T_{\x}\fo- \bar S^{*}_{ \x }\y )\;.  \]
Recall that according to Assumption \ref{basicmodeleq}, the conditional expectation w.r.t. $\rho$ of $Y$ given $X$ satisfies
\[ \E_\rho[Y | X=x]  = \bar S_x \fo\,, \] 
implying that $\xi_{\lam}$ is unbiased (since $\bar T_{\x} = \bar S^*_{\x}\bar S_{\x}$). Thus, 
\begin{equation}
\label{def:Sm}
\bar T^s( \tilde f^{\lam}_D - \bar f_D^{\lam}) \; = \; \sumup \xi_{\lam}(\x_j, \y_j)  
\end{equation} 
is a sum of centered i.i.d. random variables.

Furthermore, we need the following result from \cite{Pin94}, Theorem 5.2\, , which generalizes  Rosenthal's inequalities  from \cite{Rosen70} (originally only formulated for real valued random variables) to random variables with values in a Banach space. 
For Hilbert spaces this looks particularly nice.

\begin{prop}
\label{prop:pin94}
Let  $\hhh$ be a Hilbert space and $\xi_1, ..., \xi_m$ be a finite sequence of independent, mean zero $\hhh$- valued 
random variables. If $2 \leq p < \infty$, then there exists a constant $C_p>0$, only depending on $p$, such that
\begin{equation}
\label{eq:p-norm}   
\paren{ \E \norm{\; \sumup \xi_j \; }_{\hhh}^p   }^{\frac{1}{p}} \; \leq \; \frac{C_p}{m} \; 
           \max\left\{ \paren{\sum_{j=1}^m \E ||\xi_j||^p_{\hhh}}^{\frac{1}{p}}, \paren{\sum_{j=1}^m \E ||\xi_j||^2_{\hhh}}^{\frac{1}{2}}     \right\}    \;.   
\end{equation} 
\end{prop}  
   
We remark in passing that  \cite{Dirk11}\, , Corollary 1.22\, , contains the interesting result that in addition to the upper bound in  
(\ref{eq:p-norm}) there is  also a 
corresponding lower bound where the constant $C_p$ is replaced by
another constant $C'_p >0$, only depending on $p$.

\begin{prop}[Expectation of Sample Error]
\label{prop:expec_Sample}
Let $\rho$ be a source distribution belonging to $\mathcal{M}_{\sigma,M,R}$, $s \in [0, \frac{1}{2}]$ and let $\lam \in (0,1]$. Define $\B_{\frac{n}{m}}(\lam)$ as in 
\eqref{def:blam}.  
Assume the regularization has qualification $q \geq r+s$. For any $p \geq 1$ one has:
\[ \E_{\rho^{\otimes n}}
     \Big[\big\|\bar T^s( \tilde f^{\lam}_D - \bar f_D^{\lam})  \big\|^p_{\h}\Big]^{\frac{1}{p}} 
        \leq  C_p \; m^{-\frac{1}{2}} 
      \B_{\frac{n}{m}}(\lam)^{\frac{1}{2}+s}\lam^s\;  \left( \frac{mM}{n \lambda} + \sigma\sqrt{\frac{m {{\cal N}(\lam)}}{{n \lam }}}\right) \;, \] 
where $ C_{p}$ does not depend on $(\sigma, M,R)\in \R^3_+$.      
\end{prop}

\begin{proof}[Proof of Proposition \ref{prop:expec_Sample}]
Let $\lam \in (0,1]$ and $p \geq 2$.  From Proposition \ref{prop:pin94} 
\begin{align}
\label{eq:deco52}
 \E_{\rho^{\otimes n}}
     \left[    \snorm{ \tilde f^{\lam}_D - \bar f_D^{\lam} }^p_{\h}     \right]^{\frac{1}{p}} &= 
   \E_{\rho^{\otimes n}}\left[ \norm{   \frac{1}{m}  \sum_{j=1}^m  \xi_{\lam}(\x_j, \y_j)   }_{\h}^p  \right]^{\frac{1}{p}}   
\end{align}
\[  \leq \frac{C_p}{m} \; 
           \max\left\{ \paren{\sum_{j=1}^m \E_{\rho^{\otimes n}}\Big[ ||\xi_{\lam}(\x_j, \y_j) ||^p_{\h}\Big]}^{\frac{1}{p}}, \paren{\sum_{j=1}^m \E_{\rho^{\otimes n}}\Big[ ||\xi_{\lam}(\x_j, \y_j) ||^2_{\h}\Big]}^{\frac{1}{2}}     \right\}  \;. \]

Again, the estimates in expectation will follow from integration a bound holding with high probability. By \eqref{need}, one has for any $j=1,...,m$
\begin{align}
\label{eq:all} 
 ||\xi_{\lam}(\x_j, \y_j)||_{\h} &=  || \bar T^s g_{\lam}(\bar T_{\x_j})(\bar T_{\x_j}\fo- \bar S^{*}_{ \x_j }\y_j ) ||_{\h} \nonumber \\
  &\leq  8\log^{2s}(4\etainv) \B_{\frac{n}{m}}(\lam)^s \nonumber \\
 & \qquad ||  (\bar T_{\x_j} +\lam)^s  g_{\lam}(\bar T_{\x_j})(\bar T_{\x_j}\fo- \bar S^{*}_{ \x_j }\y_j ) ||_{\h}  \;,
\end{align}
holding with probability at least $1-\frac{\eta}{2}$, where $\B_{\frac{n}{m}}(\lam)$ is defined in \eqref{def:blam}. We proceed by  splitting:   
\[ (\bar T_{\x_j} + \lam)^s g_{\lam}(\bar T_{\x_j})(\bar T_{\x_j}\fo - \bar S_{\x_j}^{\star}\y_j )  =  H_{\x_j}^{(1)}\cdot H_{\x_j}^{(2)} \cdot h^{\lam}_{\z_j} \;, \]
with
\begin{eqnarray*}
H_{\x_j}^{(1)} &:=& (\bar T_{\x_j} + \lam)^{s} g_{\lam}(\bar T_{\x_j})(\bar T_{\x_j} + \lam)^{\frac{1}{2}} ,\\
H_{\x_j}^{(2)} &:=& (\bar T_{\x_j} + \lam)^{-\frac{1}{2}}(\bar T + \lam)^{\frac{1}{2}} , \\
h^{\lam}_{\z_j} &:=& (\bar T + \lam)^{-\frac{1}{2}} (\bar T_{\x_j}\fo - \bar S_{\x_j}^{\star}\y_j ) \;.
\end{eqnarray*}
The first term is estimated using  \eqref{eq:quali} and gives 
\begin{equation}
\label{eq:h1}
  H_{\x_j}^{(1)} \leq  C_{\blacktriangle} \lam^{s-\frac{1}{2}} \;. 
\end{equation}  
The second term is now bounded using \eqref{need} once more. One has with probability at least $1-\frac{\eta}{4}$
\begin{equation}
\label{eq:h2}
 H_{\x_j}^{(2)} \leq 8\log(8\etainv) \B_{\frac{n}{m}}(\lam)^{\frac{1}{2}} \;. 
\end{equation} 
Finally, $h^{\lam}_{\z_j}$ is estimated using Proposition \ref{Geta1}:
\begin{equation}
\label{eq:h3}
 h^{\lam}_{\z_j} \leq 2\log(8\etainv)  \left( \frac{mM}{n\sqrt{\lambda}} + \sigma\sqrt{\frac{m {{\cal N}(\lam)}}{{n}}}\right)\,, 
\end{equation}
holding with probability at least $1-\frac{\eta}{4}$. Thus, combining \eqref{eq:h1}, \eqref{eq:h2} and \eqref{eq:h3} with \eqref{eq:all} gives 
for any $j=1, ...,m$
\begin{small}
\[ ||  \xi_{\lam}(\x_j, \y_j) ||_{\h} \leq  
   C_{\blacktriangle}\log^{2(s+1)}(8\etainv) \B_{\frac{n}{m}}(\lam)^{\frac{1}{2}+s}\lam^s\;  \left( \frac{mM}{n \lambda} + \sigma\sqrt{\frac{m {{\cal N}(\lam)}}{{n \lam }}}\right)\;,\]
\end{small}
with probability at least $1-\eta$. Integration gives for any $p \geq 2$
\[ \sum_{j=1}^p \E_{\rho^{\otimes n}}
     \Big[\big\|   \xi_{\lam}(\x_j, \y_j)  \big\|^p_{\h}\Big] \leq  
     C_{\blacktriangle,p} m \;\A^p \;,\] 
with
\[ \A:= \A_{\frac{n}{m}}(\lam):=  \B_{\frac{n}{m}}(\lam)^{\frac{1}{2}+s}\lam^s\;  \left( \frac{mM}{n \lambda} + \sigma\sqrt{\frac{m {{\cal N}(\lam)}}{{n \lam }}}\right) \;.\]
Combining this with \eqref{eq:deco52} implies, since $p \geq 2$
\begin{align*}
 \E_{\rho^{\otimes n}}
     \Big[\big\|\bar T^s( \tilde f^{\lam}_D - \bar f_D^{\lam})\big\|^p_{\h}\Big]^{\frac{1}{p}} &\leq \frac{C_p}{m}\max\paren{ \paren{m\A^p}^{\frac{1}{p}}, \paren{m\A^2}^\frac{1}{2} } \\
&=  \frac{C_p}{m} \; \A \max\paren{ m^{\frac{1}{p}}, m^{\frac{1}{2}} }  \\
&= \frac{C_p}{\sqrt m} \; \A \;,
\end{align*}      
where $ C_{p}$ does not depend on $(\sigma, M,R)\in \R^3_+$. The result for the case $1\leq p\leq 2$ immediately follows from H\"older's inequality. 
\end{proof}

\begin{proof}[Proof of Theorem \ref{prop:sample_error_exp}]
Let $\lam_n$ defined by \eqref{eq:choicelam:averaged}.  According to Lemma \ref{lem:for_all}  we have $\B_{\frac{n}{m}}(\lam_n) \leq 2$ provided 
$\alpha < \frac{2br}{2br +b+1}$. We immediately obtain from Proposition  \ref{prop:expec_Sample} 
\begin{align*}
 \E_{\rho^{\otimes n}}
     \Big[\big\|\bar T^s( \tilde f^{\lam_n}_D - \bar f_D^{\lam_n})\big\|^p_{\h}\Big]^{\frac{1}{p}} &\leq \frac{C_p}{\sqrt m} \;  
      \lam_n^s\;  \left( \frac{mM}{n \lambda_n} + \sigma\sqrt{\frac{m {{\cal N}(\lam_n)}}{{n \lam_n }}}\right)  \\
  &\leq C_p \lam_n^s\;  \left( \frac{{\sqrt m}M}{n \lambda_n} + \sigma\sqrt{\frac{{{\cal N}(\lam_n)}}{{n \lam_n }}}\right) \;.
\end{align*}      
Again, we use that $\NN(\lam_n)\leq C_b \lam_n^{-1/b}$ and 
\[ \frac{\sqrt m_nM}{n\lam_n} = o\paren{ \sigma\sqrt{\frac{\lam_n^{-1/b}}{n \lam_n}} }\;, \]
provided
\[  m_n\leq n^{\alpha}\;, \qquad \alpha < \frac{2(br+1)}{2br+b+1} \;. \]
Recalling that $\sigma\sqrt{\frac{\lam_n^{-1/b}}{n \lam_n}} = R\lam^r_n = \lam_n^{-s}a_n$, we arrive at
\[ \E_{\rho^{\otimes n}}
     \Big[\big\|\bar T^s( \tilde f^{\lam_n}_D - \bar f_D^{\lam_n})\big\|^p_{\h}\Big]^{\frac{1}{p}} \leq C_p \; a_n\;  .\]
As a result, for any $1\leq p$
\[
   \limsup_{n \rightarrow \infty} \sup_{\rho \in \mathcal{M}_{\sigma,M,R}} \frac{   \E_{\rho^{\otimes n}}
     \Big[\big\| \bar T^s( \tilde f^{\lam_n}_D - \bar f_D^{\lam_n} )   \big\|^p_{\h}\Big]^{\frac{1}{p}}}{a_{n}} \leq C_{ p} \,,
\]
for some $C_{ p} < \infty$, not depending on the model parameter $(\sigma,M,R)\in \R^3_+$. 
\end{proof}


\begin{appendix}
\section{}

\begin{prop}[see e.g. \cite{BlaMuc16}]
\label{Geta2}
For any $n \in \N$, $\lambda \in (0,1]$ and $\eta \in (0,1)$, one has with probability at least $1-\eta $\,:
\begin{equation*}
\norm{(\bar T+ \lam)^{-1}(\bar T- \bar T_{\x}) }_{\hs} \; \leq 
2\log(2\eta^{-1}) \left( \frac{2}{n \lam} + \sqrt{\frac{\cal N(\lam)}{n\lam }}  \right)\; .
\end{equation*}
\end{prop}

\begin{prop}[see e.g. \cite{BlaMuc16}]
\label{Geta1}  
For $n \in \N$, $\lambda \in (0,1]$ and  $\eta \in (0,1]$, it holds with probability at least $1-\eta$\,:
\begin{equation*}
\big\| (\bar B +  \lam)^{-\frac{1}{2}}\;\left(\bar B_{\x}f_{\rho} - \bar S_{\x}^{\star}\y \right)\big\|_{\h}\;  \leq \; 
2\log(2\eta^{-1})  \left( \frac{M}{n\sqrt{\lam}} + \sqrt{\frac{\sigma^2 {\cal N}(\lam)}{ n}} \right)\;.
\end{equation*}
\end{prop}

\begin{prop}[Cordes Inequality,\cite{Bat97}, Theorem IX.2.1-2]
\label{prop:Bath-ineq}
Let $A,B$ be to self-adjoint, positive operators on a Hilbert space. Then for any $s\in[0,1]$:
\begin{equation}
\label{eq:multpert}
\norm{A^sB^s} \leq \norm{AB}^s\,.
\end{equation}
\end{prop}

\end{appendix}


\bibliography{bibliography}
\bibliographystyle{abbrv}

\checknbdrafts

\end{document}